\documentclass[a4paper, 11pt]{amsart} 
\usepackage[headings]{fullpage}               
\usepackage{boxedminipage}
\usepackage{amsfonts}
\usepackage{amsmath} 
\usepackage{amssymb}
\usepackage{graphicx}
\usepackage{amsthm}
\usepackage{subfig}
\usepackage{tikz}
\usepackage{setspace}
\usepackage{hyperref}
\usepackage{enumerate}
\hypersetup{linktocpage=true,}
\newtheorem{theorem}{Theorem}[section]
\newtheorem{lemma}[theorem]{Lemma}
\newtheorem{proposition}[theorem]{Proposition}
\newtheorem{corollary}[theorem]{Corollary}
\theoremstyle{definition}\newtheorem{definition}[theorem]{Definition}
\theoremstyle{definition}
\theoremstyle{definition}\newtheorem{remark}[theorem]{Remark}

\usepackage{etoolbox}
\let\bbordermatrix\bordermatrix
\patchcmd{\bbordermatrix}{8.75}{4.75}{}{}
\patchcmd{\bbordermatrix}{\left(}{\left[}{}{}
\patchcmd{\bbordermatrix}{\right)}{\right]}{}{}

\newcommand{\rank}{\operatorname{rank}}

\newcommand{\spann}{\operatorname{span}}
\newcommand{\smooth}{\operatorname{smooth}(X)}
\newcommand{\Iso}{\operatorname{Isom}}

\newcommand{\im}{\operatorname{im}}

\newcommand{\Isolin}{\operatorname{Isom^{Lin}}}

\newcommand{\cone}{\operatorname{cone}}

\begin{document}

\title{Infinitesimal rigidity in normed planes}

\author{Sean Dewar}

\begin{abstract}
We prove that a graph has an infinitesimally rigid placement in a non-Euclidean normed plane if and only if it contains a $(2,2)$-tight spanning subgraph. The method uses an inductive construction based on generalised Henneberg moves and the geometric properties of the normed plane. As a key step, rigid placements are constructed for the complete graph $K_4$ by considering smoothness and strict convexity properties of the unit ball.  
\end{abstract}

\maketitle
\tableofcontents

\section{Introduction}\label{Introduction}

A framework $(G,p)$ is an embedding $p$ of the vertices of a simple graph $G$ into a given normed space. With a given framework we wish to determine if it is \textit{continuously rigid} i.e.~all continuous motions of $(G,p)$ that preserve the distances between vertices joined by an edge can be extended to a rigid motion of $X$. For Euclidean spaces it was shown by L.~Asimow and B.~Roth that if $(G,p)$ is \textit{infinitesimally rigid} (rigid under infinitesimal deformations) then $(G,p)$ is continuously rigid; further, determining whether $(G,p)$ is infinitesimally rigid can be decided by matrix rank calculations \cite{asiroth} \cite{asiroth2}. In the same pair of papers, Asimow and Roth also observed that infinitesimal rigidity is a property of the graph, in the sense that either almost all embeddings of $G$ give an infinitesimally rigid framework, or none of them do. We say a graph $G$ is \textit{rigid} in $X$ if it admits an infinitesimally rigid placement in $X$, and \textit{flexible} otherwise.

In his 1970 paper \cite{laman}, G.~Laman proved we could construct all \textit{isostatic} graphs (rigid graphs with no proper spanning rigid subgraphs) from a single edge by using \textit{Henneberg moves}; the \textit{(2-dimensional) 0-extension}, where we add a vertex and connect it to two distinct vertices, and the \textit{(2-dimensional) 1-extension}, where we delete an edge and then add a vertex connected to the ends of the deleted edge and one other vertex. A key observation is that every isostatic graph must be \textit{$(2,3)$-tight}; a graph where $|E(H)| \leq 2 |V(H)|-3$ for all subgraphs $H \subset G$ with $|V(H)| \geq 2$ (the \textit{$(2,3)$-sparsity condition}) and $|E(G)| =2|V(G)|-3$ (see \cite[Theorem 5.6]{laman}). Laman then proceeded to prove the following results:

\begin{proposition}\label{Henneberg}\cite[Theorem 6.4, Theorem 6.5]{laman}
Henneberg moves preserve the $(2,3)$-tightness and $(2,3)$-sparsity of graphs. Further, any $(2,3)$-tight graph on $2$ or more vertices can be constructed from $K_2$ by a finite sequence of Henneberg moves.
\end{proposition}

\begin{proposition}\label{lamanprop}\cite[Proposition 5.3, Proposition 5.4]{laman}
If $G$ is isostatic in the Euclidean plane and $G'$ is the graph formed from $G$ by a Henneberg move then $G'$ is also isostatic in the Euclidean plane.
\end{proposition}

Combining Proposition \ref{Henneberg} and Proposition \ref{lamanprop} we obtain the following:

\begin{theorem}\label{lamanstheorem}\cite[Theorem 5.6, Theorem 6.5]{laman}
For any graph $G$ with $|V(G)| \geq 2$, $G$ is isostatic in the Euclidean plane if and only if $G$ is $(2,3)$-tight.
\end{theorem}

In this article we consider the following question: if $X$ is a \textit{non-Euclidean normed plane} (a 2-dimensional space with a norm that is not induced by an inner product) can we characterise graphs that are rigid in $X$? Framework rigidity in non-Euclidean normed spaces has been considered for $\ell_p$ normed spaces \cite{noneuclidean}, polyhedral normed spaces \cite{polyhedra} and matrix normed spaces such as the Schatten $p$-normed spaces \cite{matrixnorm}. For some normed planes we have similar results to Theorem \ref{lamanstheorem}; for example a graph $G$ is isostatic in any $\ell_p$ plane ($p \neq 2$) or any polyhedral normed plane if and only if $G$ is \textit{$(2,2)$-tight} i.e.~ $|E(H)| \leq 2 |V(H)|-2$ for all subgraphs $H \subset G$ (the \textit{$(2,2)$-sparsity condition}) and $|E(G)| =2|V(G)|-2$ \cite{noneuclidean} \cite{polyhedra}. If $G$ is isostatic in any non-Euclidean normed plane, $G$ will be $(2,2)$-tight (see part \ref{papernecessary2item3} of Theorem \ref{papernecessary2}). In fact, these $(2,2)$-tight graphs are exactly the rigid graphs for any non-Euclidean normed plane, which we prove with the following result:

\begin{theorem}\label{laman2}
Let $X$ be a non-Euclidean normed plane. Then a graph $G$ is isostatic in $X$ if and only if $G$ is $(2,2)$-tight.
\end{theorem}

To prove Theorem \ref{laman2} we employ a similar method to Laman, however, we require two additional graph operations: \textit{vertex splitting} (see Section \ref{Edge-to-K3 extensions}) and \textit{vertex-to-$K_4$ extensions} (see Section \ref{Vertex-to-K4 extensions}). These graph operations were originally applied in the context of infinitesimal rigidity in \cite{vertexsplit} and \cite{twotwoparttwo} respectively. The following result provides an analogue for Proposition \ref{Henneberg}.

\begin{proposition}\label{nonHenneberg}\cite[Theorem 1.5]{twotwoparttwo}
Henneberg moves, vertex splitting and vertex-to-$K_4$ extensions preserve $(2,2)$-tightness and $(2,2)$-sparsity. Further, if $G$ is $(2,2)$-tight then it may constructed from $K_1$ by a finite sequence of Henneberg moves, vertex splitting and vertex-to-$K_4$ extensions.
\end{proposition}

To complete our characterisation we need an analogue of Proposition \ref{lamanprop}. Of the four graph operations, the vertex to K4 proves the most challenging. In particular, we must first establish that K4 is isostatic in any non-euclidean normed plane. 

The structure of the paper will be as follows.

In Section \ref{Framework and graph rigidity} we shall lay out some of the basic definitions and results for graph rigidity in non-Euclidean normed spaces. We shall also develop many of the tools we will need to prove Theorem \ref{laman2}, such as how to approximate frameworks with non-differentiable edge-distances with frameworks with differentiable edge-distances.

In Section \ref{Rigidity of K4 in all normed planes} we shall prove that $K_4$ is rigid in all normed planes. To do this we shall split into three cases dependent on whether the normed plane $X$ is \textit{smooth} (the norm of $X$ is differentiable at every non-zero point) or \textit{strictly convex} (the unit ball of $X$ is strictly convex). The cases will be:
\begin{enumerate}[(i)]
\item $X$ is not strictly convex,
\item $X$ is strictly convex but not smooth, 
\item $X$ is both strictly convex and smooth. 
\end{enumerate}
For the first case we will construct an infinitesimally rigid placement of $K_4$ that takes advantage of the lack of strict convexity. In the second case we shall construct a sequence of placements $p^n$ of $K_4$ and show that $(K_4,p^n)$ will be infinitesimally rigid for large enough $n$. In the last case we shall use methods utilised in \cite{rotation} to prove the existence of an infinitesimally rigid placement of $K_4$.

In Section \ref{Graph extensions in the normed plane} we shall define the required graph operations that we need and show each move preserves graph isostaticity in non-Euclidean normed planes. 

In Section \ref{Generalised Laman and Lovasz and Yemini theorems} we shall prove Theorem \ref{laman2} using the results from Section \ref{Rigidity of K4 in all normed planes} and Section \ref{Graph extensions in the normed plane}, and we shall give some immediate corollaries to the result. We shall also use Theorem \ref{laman2} to give some sufficient connectivity conditions for graph rigidity analogous to those given by Lov\'{a}sz \& Yemini for the Euclidean plane in \cite{lovasz}.

\section{Preliminaries}\label{preliminaries}

All normed spaces $(X, \| \cdot \|)$ shall be assumed to be over $\mathbb{R}$ and finite dimensional; further we shall denote a normed space by $X$ when there is no ambiguity. For any normed space $X$ we shall use the notation $B^X_r(x)$, $B^X_r[x]$ and $S^X_r[x]$ for the open ball, closed ball and the sphere with centre $x \in X$ and radius $r > 0$ respectively. When it is clear what normed space we are talking about we shall drop the $X$; if the normed space is the dual space $X^*$ we shall shorten to $B^*_r[f]$, $B^*_r(f)$ and $S^*_r[f]$ for any $f \in X^*$ and $r >0$. For any $x_1, x_2 \in X$ we denote by
\begin{align*}
[x_1,x_2] := \{tx_1 +(1-t)x_2 : t \in [0,1] \} \qquad (x_1,x_2) := \{tx_1 +(1-t)x_2 : t \in (0,1) \}.
\end{align*}
the \textit{closed line segment (for $x_1,x_2$)} and \textit{open line segment (for $x_1,x_2$)} respectively.

Given normed spaces $X,Y$ we shall denote by $L(X,Y)$ the normed space of all linear maps from $X$ to $Y$ with the operator norm $\|\cdot\|_{\text{op}}$ and $A(X, Y)$ to be space of all affine maps from $X$ to $Y$ with the norm topology. If $X=Y$ we shall abbreviate to $L(X)$ and $A(X)$ and if $Y= \mathbb{R}$ with the standard norm we define $X^*:= L(X, \mathbb{R})$ and refer to the operator norm as $\| \cdot \|$ when there is no ambiguity. We denote by $\iota$ the identity map on $X$.

\subsection{Support functionals, smoothness and strict convexity} \label{Support functionals, smoothness and strict convexity}

Let $x \in X$ and $f \in X^*$, then we say that $f$ is \textit{support functional} of $x$ if $\|f\| = \|x\|$ and $f(x) = \|x\|^2$. By an application of the Hahn-Banach theorem it can be shown that every point must have a support functional. 

We say that a non-zero point $x$ is \textit{smooth} if it has a unique support functional and define $\smooth \subseteq X \setminus \{0\}$ to be the set of smooth points of $X$. 

The \textit{dual map} of $X$ is the map $\varphi : \smooth \cup \{0\} \rightarrow X^*$ that sends each smooth point to its unique support functional and $\varphi(0)=0$. It is immediate that $\varphi$ is homogeneous since $f$ is the support functional of $x$ if and only if $a f$ is the support functional of $ax$ for $a \neq 0$.

\begin{remark}
If $X$ is Euclidean with inner product $\left\langle \cdot , \cdot \right\rangle$ then all non-zero points are smooth and we have $\varphi(x) = \left\langle x, \cdot \right\rangle$ where $\left\langle x, \cdot \right\rangle : y \mapsto \left\langle x, y \right\rangle$.
\end{remark}

\begin{proposition}\cite[Proposition 2.3]{mypaper}\label{paper1}
For any normed space $X$ the following properties hold:
\begin{enumerate}[(i)]
\item \label{paper1item1} For $x_0 \neq 0$, $x_0 \in \smooth$ if and only if $x \mapsto \|x\|$ is differentiable at $x_0$.

\item \label{paper1item3} If $x \mapsto \|x\|$ is differentiable at $x_0$ then it has derivative $\frac{1}{\|x_0\|}\varphi(x_0)$.

\item \label{paper1item0} The set $\smooth$ is dense in $X$ and $\smooth^c$ has Lebesgue measure zero with respect to the Lebesgue measure on $X$.

\item \label{paper1item2} The map $\varphi$ is continuous.
\end{enumerate}
\end{proposition}

If $\smooth \cup \{0\} =X$ then we say that $X$ is \textit{smooth}. We define a norm to be \textit{strictly convex} if $\|tx +(1-t) y\| < 1$ for all distinct $x,y \in S_1[0]$ and $t \in (0,1)$. The following is a useful property of strictly convex spaces.

\begin{proposition}\label{str}
Let $X$ be strictly convex then the following hold:
\begin{enumerate}[(i)]
\item \label{stritem1} $\varphi$ is injective.

\item \label{stritem2} If $x,y \in X$ are linearly independent then $\varphi(x), \varphi(y)$ are linearly independent.
\end{enumerate}
\end{proposition}

\begin{proof}
(\ref{stritem1}): Suppose $\varphi(x) = \varphi(y)$ for $x \neq y$, then $\| x \|=\|y\|$; as $\varphi$ is homogenous we may assume without loss of generality that $\|x\|= \|y\|=1$. For all $t \in (0,1)$ we have 
\begin{align*}
1= t\varphi(x)x +(1-t)\varphi(y) y =\varphi(x)(tx+(1-t)y) \leq \|tx +(1-t)y\| ,
\end{align*}
thus $X$ is not strictly convex. 

(\ref{stritem2}): Suppose $\varphi(x), \varphi(y)$ are linearly dependent, then $\varphi(x) = c\varphi(y)$ for some $c \in \mathbb{R}$. As $\varphi$ is homogenous it follows $\varphi(x) = \varphi(c y)$, thus by part \ref{stritem1}, $x = cy$ as required.
\end{proof}

As every point has at least one support functional we shall define for each $x \in X$ the set $\varphi[x]$ of support functionals of $x$; note that $x$ is smooth if and only if $|\varphi[x]|=1$.

\begin{proposition}\label{supportset}
For any $x \in X \setminus \{0\}$ the following holds:
\begin{enumerate}[(i)]
\item \label{supportsetitem1} $\varphi[x]$ is a compact and convex subset of $S^*_{\|x\|}[0]$.

\item \label{supportsetitem2} If $\dim X =2$ then $\varphi[x] =[f,g]$ for some $f,g \in X^*$ and $x \in \smooth$ if and only if $f = g$.
\end{enumerate}
\end{proposition}

\begin{proof}
(\ref{supportsetitem1}): For each $f \in \varphi[x]$ we have $\|f \|=\|x\|$ by definition thus $\varphi[x] \subset S^*_{\| x\|}[0]$. Given $f,g \in \varphi[x]$ and $t \in [0,1]$ we note that $(t f +(1-t)g)(x) = \|x\|^2$ and 
\begin{align*}
\|t f +(1-t)g\| \leq t\|f \| + (1-t)\|g\| = \|x\|,
\end{align*}
thus $t f +(1-t)g \in \varphi[x]$ and $\varphi[x]$ is convex. Finally if $(f_n)_{n \in \mathbb{N}}$ is a convergent sequence of support functionals of $x$ with limit $f$ then $\|f\|= \|x\|$ and $f(x) = \lim_{n \rightarrow \infty}f_n(x) = \|x\|^2$ thus $f \in \varphi[x]$; since this implies $\varphi[x]$ is a closed subset of the compact set $S^*_{\|x\|}[0]$ then it too is compact.

(\ref{supportsetitem2}): If $x$ is smooth then $\varphi[x] = \{ \varphi(x) \} = [ \varphi(x), \varphi(x)]$. Suppose $x$ is not smooth, then by \ref{supportsetitem1}, $\varphi[x]$ is a compact convex subset of the $1$-dimensional manifold $S^*_{\|x\|}[0]$, and hence is a line segment.
\end{proof}

We define for $S_1[0]$ the \textit{(inner) L\"{o}wner ellipsoid} $S$ of $S_1[0]$, the unique convex body of maximal volume bounded by $S_1[0]$ which has a Minkowski functional $\| \cdot \|_S: X \rightarrow \mathbb{R}_{\geq 0}$ that can be induced by an inner product. It is immediate that $\|x\|_S \geq \|x\|$ for all $x \in X$ and the Euclidean space $(X,\| \cdot \|_S)$ has unit sphere $S$. For more information on L\"{o}wner ellipsoids see \cite[Chapter 3.3]{minkowski}.

\begin{proposition}\label{abc}
Suppose $\dim X \geq 2$. For all $x \in S_1[0] \cap \smooth$ there exists $y \in S_1[0] \cap \smooth$ such that $x \neq y$ and $\varphi(x), \varphi(y)$ are linearly independent.
\end{proposition}

\begin{proof}
By \cite[Lemma 6.1]{euclidean} there exists $y_1, \ldots, y_d \in S_1[0]$ that lie on the L\"{o}wner ellipsoid $S$ of $S_1[0]$. Suppose $f_i$ is a support functional for $y_i$ with respect to $\| \cdot \|$ and choose any $x \in S$. As $S \subset B_1[0]$ (the unit ball of $(X,\| \cdot \|)$) then $|f_i(x)| \leq 1$, thus $f$ is a support functional for $y_i$ with respect to $\| \cdot \|_S$ also. As $(X, \| \cdot \|_S)$ is Euclidean then it follows that $y_1, \ldots, y_d$ are smooth and $\varphi(y_1), \ldots, \varphi(y_d)$ are linearly independent.

If $x = y_i$ for some $i=1, \ldots, d$ then there exists $j\neq i$ such that $x \neq y_j$ and we let $y  = y_j$. If $x \neq y_i$ for all $i = 1, \ldots, d$ then $\varphi(x)$ has to be linearly independent to some $\varphi(y_i)$ and we let $y = y_i$.
\end{proof}

\subsection{Isometries of Euclidean and non-Euclidean planes} \label{Isometries of Euclidean and non-Euclidean planes}

We shall define $\Iso (X, \| \cdot \|)$ to be the \textit{group of isometries} of $(X, \| \cdot \|)$ and $\Isolin (X, \| \cdot \|)$ to be the \textit{group of linear isometries} of $X$ with the group actions being composition; we shall denote these as $\Iso (X)$ and $\Isolin (X)$ if there is no ambiguity. It can be seen by Mazur-Ulam's theorem \cite{minkowski} that all isometries of a finite dimensional normed space are affine i.e.~each isometry is the unique composition of a linear isometry followed by a translation, thus $\Iso (X)$ has the topology inherited from $A(X)$. 

It follows from the Closed Subgroup theorem \cite[Theorem 5.1.14]{manifold} that for any normed space the group of isometries is a \textit{Lie group} (a smooth finite dimensional manifold with smooth group operations) while the group of linear isometries is a compact Lie group since it is closed and bounded in $L(X)$. We denote by $T_\iota \Iso (X)$ the tangent space of the smooth manifold $\Iso (X)$ at the identity map $\iota:X \rightarrow X$.

For 2-dimensional normed spaces we can immediately categorize $\Iso (X)$ into one of two possibilities.

\begin{proposition}\label{paperiso1}
Let $X$ be a normed plane, then the following holds:
\begin{enumerate}[(i)]
\item \label{paperiso1item1} If $X$ is Euclidean then there are infinitely many linear isometries of $X$ and $T_\iota \Iso (X) = \spann \{ T_1, T_2, T_0\}$ where $T_1, T_2$ are linearly independent translations and $T_0$ is a linear map.

\item \label{paperiso1item2} If $X$ is non-Euclidean then there are a finite amount of linear isometries of $X$ and $T_\iota \Iso (X) = \spann \{ T_1, T_2\}$ where $T_1, T_2$ are linearly independent translations. 
\end{enumerate}
\end{proposition}

\begin{proof}
(\ref{paperiso1item1}): As all Hilbert spaces of the same dimension are isometrically isomorphic then $X$ is isometrically isomorphic to the Euclidean plane and the result follows.

(\ref{paperiso1item2}): As remarked in \cite[pg. 83]{minkowski} there are only finitely many linear isometries $\iota:=L_0$, $L_1, \ldots, L_n$ of $X$ and so by Mazur-Ulam's theorem \cite[Theorem 3.1.2]{minkowski} we have
\begin{align*}
\Iso(X) = \{ T_x \circ L_i : x \in X, ~ i=0, \ldots, n \}
\end{align*}
where $T_x(y) = x+y$ for all $y \in X$. We can now see that the tangent space at $\iota$ is exactly the space of translations and the result follows.
\end{proof}

\section{Framework and graph rigidity} \label{Framework and graph rigidity}

\subsection{Frameworks} \label{Frameworks}

We shall assume that all graphs are finite and simple i.e.~no loops or parallel edges. We will denote $V(G)$ and $E(G)$ to be the vertex and edge sets of $G$ respectively. If $H$ is a subgraph of $G$ we will represent this by $H \subseteq G$. For a set $S$ we shall denote by $K_S$ the complete graph on the set $S$; alternatively we will denote $K_n$ to be the complete graph on $n$ vertices ($n \in \mathbb{N}$). For any set $S \subset V(G)$ we denote by $G[S]$ the induced subgraph of $G$ on $S$. 

Let $X$ be a normed space. We define a \textit{(bar-joint) framework} to be a pair $(G,p)$ where $G$ is a graph and $p \in X^{V(G)}$; we shall refer to $p$ as a \textit{placement of $G$}. For all $X$ and $G$ we will gift $X^{V(G)}$ and $\mathbb{R}^{E(G)}$ the following norms:
\begin{align*}
\|\cdot \|_{V(G)} : (x_v)_{v \in V(G)} \mapsto \max_{v \in V(G)} \|x_v\| \qquad  \|\cdot \|_{E(G)} : (a_e)_{e \in E(G)} \mapsto \max_{e \in E(G)} \|a_e\|.
\end{align*}
For $x \in X^{V(G)}$, $a \in \mathbb{R}^{E(G)}$, and $H \subset G$ we define $x|_{H} := (x_v)_{v \in V(H)} \in X^{V(H)}$ and $a|_{H} := (a_e)_{e \in E(H)} \in \mathbb{R}^{E(H)}$.

A placement $p$ is in \textit{general position} if for any choice of distinct vertices $v_0, v_1, \ldots, v_n \in V(G)$ ($n \leq \dim X$) the set $\{p_{v_i} : i = 0,1 ,\ldots,n\}$ is affinely independent. For any graph $G$ we let $\mathcal{G}(G)$ be the set of placements of $G$ in general position. As $\mathcal{G}(G)$ is the complement of an algebraic set then $\mathcal{G}(G)$ is an open dense subset of $X^{V(G)}$ and $\mathcal{G}(G)^c$ has measure zero with respect to the Lebesgue measure of $X^{V(G)}$.

For frameworks $(H,q)$ and $(G,p)$ we say $(H,q)$ is a \textit{subframework} of $(G,p)$ (or $(H,q) \subseteq (G,p)$) if $H \subseteq G$ and $p_v = q_v$ for all $v \in V(H)$. If $H$ is also a spanning subgraph we say that $(H,q)$ is a \textit{spanning subframework} of $(G,p)$.

\subsection{The rigidity map and rigidity matrix} \label{The rigidity map and rigidity matrix}

We say that an edge $vw \in E(G)$ of a framework $(G,p)$ is \textit{well-positioned} if $p_v -p_w \in \smooth$; if this holds we define $\varphi_{v,w} := \varphi \left(\frac{p_v-p_w}{\|p_v-p_w\|} \right)$ to be the \textit{support functional of $vw$ for $p$}. If our placement has a superscript, i.e.~$p^\delta$, then we will define $\varphi^\delta_{v,w}$ to be the support functional of the edge $vw$ for $p^\delta$ (if it is well-positioned). If all edges of $(G,p)$ are well-positioned we say that $(G,p)$ is \textit{well-positioned} and $p$ is a \textit{well-positioned placement} of $G$. We shall denote the subset of well-positioned placements of $G$ in $X$ by the set $\mathcal{W}(G)$. 

\begin{lemma}\cite[Lemma 4.1]{mypaper}\label{wellpos}
The set $\mathcal{W}(G)$ is dense in $X^{V(G)}$ and $\mathcal{W}(G)^c$ has measure zero with respect to the Lebesgue measure of $X^{V(G)}$.
\end{lemma}

We can extend this result to placements where we fix some subset of points.

\begin{lemma}\label{wellpos2}
Let $\emptyset \neq V \subsetneq V(G)$ and $p \in X^V$ chosen such that $p_v - p_w \in \smooth$ for all $vw \in E(G)$, $v,w \in V$. Then the set
\begin{align*}
\mathcal{W}(G)_V := \{ q \in X^{V(G) \setminus V} : q \oplus p \in \mathcal{W}(G) \}
\end{align*}
is dense in $X^{V(G) \setminus V}$ and $\mathcal{W}(G)_V^c$ has measure zero with respect to the Lebesgue measure of $X^{V(G) \setminus V}$.
\end{lemma}

\begin{proof}
If $G$ has one edge the result can be seen to immediately follow from part \ref{paper1item0} of Lemma \ref{paper1}. Suppose this holds for all graphs with $n-1$ edges and let $G$ be a graph with $n$ edges. If there exists no edge connecting $V$ and $V(G) \setminus V$ then $\mathcal{W}(G)_V = \mathcal{W}(G[V(G)\setminus V])$ and so the result follows from Lemma \ref{wellpos}. Suppose there exists $vw \in E(G)$ such that $v \in V$ and $w \in V(G) \setminus V$. Define $G_1,G_2$ to be the subgraphs of $G$ where $V(G_1)=V(G_2)=V(G)$, $E(G_1) := E(G) \setminus \{vw\}$ and $E(G_2) := \{vw\}$. By assumption $\mathcal{W}(G_1)_V^c$ and $\mathcal{W}(G_2)_V^c$ have measure zero, thus as $\mathcal{W}(G)_V^c = \mathcal{W}(G_1)_V^c \cup \mathcal{W}(G_2)_V^c$ then it too has measure zero. As the complement of a measure zero set is dense the result follows by induction.
\end{proof}

We define the \textit{rigidity operator} of $G$ at $p$ in $X$ to be the continuous linear map
\begin{align*}
df_G(p): X^{V(G)} \rightarrow \mathbb{R}^{E(G)}, ~ x = (x_v)_{v \in V(G)} \mapsto (\varphi_{v,w}(x_v - x_w ))_{vw \in E(G)}.
\end{align*}

\begin{lemma}\cite[Lemma 4.3]{mypaper}\label{rigopcont}
The map
\begin{align*}
df_G : \mathcal{W}(G) \rightarrow L(X^{V(G)}, \mathbb{R}^{E(G)}), ~ x \mapsto df_G(x)
\end{align*}
is continuous.
\end{lemma}

We say that a well-positioned framework $(G,p)$ is \textit{regular} if for all $q \in \mathcal{W}(G)$ we have $\rank df_G(p) \geq \rank df_G(q)$. We shall denote the subset of $\mathcal{W}(G)$ of regular placements of $G$ by $\mathcal{R}(G)$.

\begin{lemma}\cite[Lemma 4.4]{mypaper}\label{regopen2}
The set $\mathcal{R}(G)$ is a non-empty open subset of $\mathcal{W}(G)$.
\end{lemma}

\begin{lemma}\label{regopen3}
The set $\mathcal{R}(G) \cap \mathcal{G}(G)$ is a non-empty open subset of $\mathcal{W}(G)$.
\end{lemma}

\begin{proof}
By Lemma \ref{wellpos}, $\mathcal{W}(G)^c$ has measure zero. As $\mathcal{G}(G)^c$ is an algebraic set then it is closed with measure zero, thus $\mathcal{G}(G) \cap \mathcal{W}(G)$ is dense in $X^{V(G)}$ and $\mathcal{G}(G) \cap \mathcal{W}(G)$ is an open dense subset of $\mathcal{W}(G)$. By Lemma \ref{regopen2}, $\mathcal{R}(G)$ is open in $\mathcal{W}(G)$ and so the result follows.
\end{proof}

For any well-positioned framework we can define the \textit{rigidity matrix of $(G,p)$ in $X$} to be the $|E(G)| \times |V(G)|$ matrix $R(G,p)$ with entries in the dual space $X^*$ given by
\begin{align*}
a_{e,v} := 
\begin{cases}
\varphi_{v,w}, & \text{if } e = vw \in E(G) \\
0, & \text{otherwise}
\end{cases}
\end{align*}
for all $(e,v) \in E(G) \times V(G)$. 

For any $|E(G)| \times |V(G)|$ matrix $A$ with entries in the dual space $X^*$ we may regard $A$ as the linear transform from $X^{V(G)}$ to $\mathbb{R}^{E(G)}$ given by
\begin{align*}
u \mapsto A (u) := \left( \sum_{w' \in V(G)} a_{(vw,w')}(u_{w'}) \right)_{vw \in E(G)}.
\end{align*}
By this definition we see that $A$ has row independence if and only if $A$ is surjective when considered as a linear transform. With this definition we note that $R(G,p)$ is a matrix representation of $df_G(p)$; we shall often use the notation $R(G,p)$ if we wish to observe properties involving the structure of the matrix and $df_G(p)$ if we wish to observe properties of the linear map.

\subsection{Infinitesimal rigidity and independence of frameworks} \label{Infinitesimal rigidity and independence of frameworks}

We define $u \in X^{V(G)}$ to be a \textit{trivial (infinitesimal) motion of $p$} if there exists $g \in T_\iota \Iso (X)$ such that $(g(p_v))_{v \in V(G)} =u$. For any placement $p$ we shall denote $\mathcal{T}(p)$ to be the the set all trivial infinitesimal motions of $p$. 

If $(G,p)$ is well-positioned we say that $u \in X^{V(G)}$ is an \textit{(infinitesimal) flex of $(G,p)$} if $df_G(p) u =0$; we will denote by $\mathcal{F}(G,p)$ the set of all infinitesimal flexes of $(G,p)$. The set $\mathcal{F}(G,p)$ is clearly a linear space as it is exactly the kernel of the rigidity operator. By \cite[Lemma 4.5]{mypaper} it follows $\mathcal{T}(p) \subseteq \mathcal{F}(G,p)$. We define a flex to be \textit{trivial} if it is also a trivial motion of its placement.

A well-positioned framework $(G,p)$ is \textit{infinitesimally rigid (in $X$)} if every flex is trivial and \textit{infinitesimally flexible (in $X$)} otherwise. We shall define a well-positioned $(G,p)$ framework to be \textit{independent} if the rigidity operator of $G$ at $p$, $df_G(p)$, is surjective and define $(G,p)$ to be \textit{dependent} otherwise. If a framework is infinitesimally rigid and independent we shall say that it is \textit{isostatic}. We shall use the convention that any framework with no edges (regardless of placement) is independent; this will include the \textit{null framework} $(K_0, p )$ where $V(K_0)=E(K_0)= \emptyset$. We note that $(K_1,p)$ is rigid and so with our assumptions it will be isostatic.

We have a few equivalent definitions for independence. We first define for any well-positioned framework $(G,p)$ an element $(a_{vw})_{vw \in E(G)} \in \mathbb{R}^{E(G)}$ to be a \textit{stress} of $(G,p)$ if it satisfies the \textit{stress condition} at each vertex $v \in V(G)$, i.e.
\begin{align*}
\sum_{w \in N_G(v)} a_{vw} \varphi_{v,w} =0.
\end{align*}

\begin{proposition}\label{equivind}
For any well-positioned framework the following are equivalent:
\begin{enumerate}[(i)]

\item \label{indyitem1} $(G,p)$ is independent.

\item \label{indyitem2} $R(G,p)$ has independent rows.

\item \label{indyitem3} $|E(G)| = \rank df_G(p)$.

\item \label{indyitem4} The only stress of $(G,p)$ is the zero stress i.e.~$a_{vw}=0$ for all $vw \in E(G)$.
\end{enumerate}
\end{proposition}

\begin{proof}
(\ref{indyitem1} $\Leftrightarrow$ \ref{indyitem2}): If we consider $R(G,p)$ as a linear transform then it is surjective if and only if it has row independence. As $R(G,p) = df_G(p)$ when considered as a linear transform the result follows.

(\ref{indyitem1} $\Leftrightarrow$ \ref{indyitem3}): This follows immediately as $\im df_G(p) \subseteq \mathbb{R}^{E(G)}$.

(\ref{indyitem2} $\Leftrightarrow$ \ref{indyitem4}): A non-zero stress is equivalent to a linear dependence on the edges of $R(G,p)$.
\end{proof}

\begin{remark}
Let $(G,p)$ be a well-positioned framework, then we may define a subset $E \subset E(G)$ to be independent if the subframework generated on the edge set $E$ is an independent framework. Since framework independence is a property determined by matrix row independence then the power set of $E(G)$ with the independent sets as defined will be a matroid.
\end{remark}

The following gives us some necessary and sufficient conditions for infinitesimal rigidity.

\begin{theorem}\cite[Theorem 10]{maxwell}\label{maxwellpaper}
Let $(G,p)$ be well-positioned in $X$, then the following hold:
\begin{enumerate}[(i)]
\item $(G,p)$ is independent $\Rightarrow$ $|E(G)| = (\dim X) |V(G)| - \dim \mathcal{F}(G,p)$.

\item $(G,p)$ is infinitesimally rigid $\Rightarrow$ $|E(G)| \geq (\dim X) |V(G)| - \dim \mathcal{T}(p)$.
\end{enumerate}
\end{theorem}

\begin{corollary}\cite[Lemma 4.11]{mypaper}\label{papernecessary}
Let $(G,p)$ be a independent framework with $|V(G)|\geq \dim X+1$. Then for all $H \subset G$ with $|V(H)| \geq \dim X+1$ we have $|E(H)| \leq (\dim X) |V(H)| - \dim \Iso (X)$. If $(G,p)$ is isostatic then $|E(G)|=(\dim X)|V(G)|-\dim \Iso (X)$.
\end{corollary}

The following gives an equivalence for isostaticity.

\begin{proposition}\label{isostatic}
Let $(G,p)$ be a well-positioned framework in $X$. If any two of the following properties hold then so does the third (and $(G,p)$ is isostatic):
\begin{enumerate}[(i)]
\item $|E(G)| = (\dim X)|V(G)| - \dim \mathcal{T}(p)$
\item $(G,p)$ is infinitesimally rigid
\item $(G,p)$ is independent. 
\end{enumerate}
\end{proposition}

\begin{proof}
Apply the Rank-Nullity theorem to the rigidity operator of $G$ at $p$. The result follows the same method as \cite[Lemma 2.6.1.c]{comrig}.
\end{proof}

\subsection{Rigidity and independence of graphs in the plane} \label{Rigidity and independence of graphs}

Let $G$ be any graph and $k \in \{2,3\}$. If for all $H \subseteq G$ we have that
\begin{align*}
|E(H)| \leq \max \{2 |V(H)| - k, 0 \}
\end{align*}
then we say that $G$ is \textit{$(2,k)$-sparse}. If $G$ is $(2,k)$-sparse and 
\begin{align*}
|E(G)| = 2 |V(G)| - k
\end{align*}
then we say that $G$ is \textit{$(2,k)$-tight}.

We shall say a graph $G$ is \textit{rigid in $X$} if there exists $p \in X^{V(G)}$ such that $(G,p)$ is infinitesimally rigid. Likewise we shall define a graph to be \textit{independent in $X$} if there exists an independent placement of $G$ and \textit{isostatic in $X$} if there exists an isostatic placement of $G$. 

\begin{theorem}\label{papernecessary2}
Let $X$ be a normed plane. We let $k =3$ if $X$ is Euclidean and $k=2$ if $X$ is non-Euclidean. For any graph $G$ with at least two vertices the following holds:
\begin{enumerate}[(i)]
\item \label{papernecessary2item1} If $|V(G)| \leq 3$ then $G$ is rigid if and only if $X$ is Euclidean and $G = K_2$ or $K_3$.

\item \label{papernecessary2item2} If $G$ is independent then $G$ is $(2,k)$-sparse.

\item \label{papernecessary2item3} If $G$ is isostatic $G$ is $(2,k)$-tight.

\item \label{papernecessary2item4} If $G$ is rigid then $G$ contains a $(2,k)$-tight spanning subgraph.
\end{enumerate}

\end{theorem}

\begin{proof}
We note $k = \Iso (X)$ by Proposition \ref{paperiso1}.

(\ref{papernecessary2item1}): This follows from \cite[Propositin 5.7]{mypaper} and \cite[Theorem 5.8]{mypaper}.

(\ref{papernecessary2item2}) \& (\ref{papernecessary2item3}): Suppose $|V(G)| \geq 3$. Let $(G,p)$ be an independent placement of $G$, then any subframework $(H,q)$ is also independent. If $|V(H)| \leq 2$ then $H$ is $(2,k)$-tight, thus by Corollary \ref{papernecessary} applied to any such subframework $(H,q)$ we have that $G$ is $(2,k)$-sparse and $(2,k)$-tight if it is isostatic. Suppose $|V(G)|=2$, then no graph us isostatic if $X$ is non-Euclidean by \ref{papernecessary2item1}. If $X$ is Euclidean then we note that $K_2$ is the only isostatic graph and it is $(2,3)$-tight.

(\ref{papernecessary2item4}): This follows from \ref{papernecessary2item3}.
\end{proof}

\begin{corollary}\label{isostaticgraphs}
Let $X$ be a normed plane and $k:= \Iso (X)$. For any graph $G$ with at least 3 vertices, if two of the following hold so does the third (and $G$ is isostatic):
\begin{enumerate}[(i)]
\item $|E(G)| = 2|V(G)| - k$

\item $G$ is independent

\item $G$ is rigid.
\end{enumerate}
\end{corollary}

\begin{proof}
By Lemma \ref{regopen3} we may choose a regular placement $p$ of $G$ in general position. By \cite[Corollary 3.10]{mypaper} and \cite[Theorem 3.14]{mypaper}, $\dim \mathcal{T}(p) = \dim \Iso (X)$. We now apply Proposition \ref{isostatic}.
\end{proof}

\begin{remark}
We note that any framework in a non-Euclidean normed plane will be \textit{full} by part \ref{paperiso1item2} of Proposition \ref{paperiso1}, i.e.~for any placement $p$ we have $\dim \mathcal{T}(p) = \dim \Iso (X) =2$. For an in-depth discussion on the topic see \cite{mypaper}.
\end{remark}

\subsection{Pseudo-rigidity matrices and approximating not well-positioned frameworks} \label{Approximation by well-positioned frameworks}

Often frameworks which are not well-positioned can be used to obtain information about well-positioned frameworks. We can apply the following method to test for independence, mainly applied in sections \ref{The rigidity of K4 in strictly convex but not smooth normed planes} and \ref{Graph extensions in the normed plane}.

Suppose $(G,p)$ is a not well-positioned framework in a normed space $(X,\| \cdot\|)$ with an open set of smooth points, then there exists a non-empty subset $F \subset E(G)$ of non-well-positioned edges. For each $vw \in F$ we will choose some $f \in X^*$ and define $\varphi_{v,w} := f$. We define $\varphi_{v,w}$ to be the \textit{pseudo-support functional of $vw$ for $p$}. Using the support functionals of the edges in $E(G) \setminus F$ and the chosen pseudo-support functionals of the edges in $F$ we define $\phi := \{\varphi_{v,w} :vw \in E(G) \}$ to be the set of support functionals and pseudo-support functionals for our framework and $R(G,p)^\phi$ to be the $|E(G)| \times |V(G)|$ \textit{pseudo-rigidity matrix} generated by our set $\phi$ in the same manner as described in Section \ref{The rigidity map and rigidity matrix}. We shall also use the notation $(G,p)^\phi$ to indicate that we are considering $(G,p)$ with the pseudo-rigidity matrix $R(G,p)^\phi$.

We define $(G,p)^\phi$ to be \textit{independent} if $R(G,p)^\phi$ has row independence and \textit{dependent} otherwise. We define a vector $a := (a_{vw})_{vw \in E(G)} \in \mathbb{R}^{E(G)}$ to be a \textit{pseudo-stress of $(G,p)^\phi$} if it satisfies the \textit{pseudo-stress condition} i.e.~for all $v \in V(G)$, $\sum_{w \in N_G(v)} a_{vw} \varphi_{v,w} =0$. Following from Proposition \ref{equivind} we can see that $(G,p)^\phi$ is independent if and only if the only pseudo-stress is $(0)_{vw \in E(G)}$.

Suppose we have a sequence $(p^n)_{n \in \mathbb{N}}$ of well-positioned placements of $G$ such that $p^n \rightarrow p$ as $n \rightarrow \infty$ and the sequences $(\varphi^n_{v,w})_{n \in \mathbb{N}}$ in $X^*$ converge for all $vw \in E(G)$, where $\varphi^n_{v,w}$ is the support functional of $vw$ in $(G,p^n)$. If $vw \in E(G) \setminus F$ then by part \ref{paper1item2} of Proposition \ref{paper1}, $\varphi^n_{v,w} \rightarrow \varphi_{v,w}$ as $n \rightarrow \infty$. We say that $(G,p)^\phi$ is the \textit{framework limit} of $(G,p^n)$ (or $(G,p^n) \rightarrow (G,p)^\phi$ as $n \rightarrow \infty$) if $\varphi^n_{v,w} \rightarrow \varphi_{v,w}$ for all $vw \in E(G)$.

\begin{proposition}\label{framelimit}
Suppose $(G,p)^\phi$ is the framework limit of the sequence of well-positioned frameworks $((G,p^n))_{n \in \mathbb{N}}$ in $X$. If $R(G,p)^\phi$ has row independence then there exists $N \in \mathbb{N}$ such that $(G,p^n)$ is independent for all $n \geq N$.
\end{proposition}

\begin{proof}
First note that if we consider $|E(G)| \times |V(G)|$ matrices with entries in $X^*$ to be elements of $L(X^{V(G)}, \mathbb{R}^{E(G)})$ as described in Section \ref{The rigidity map and rigidity matrix} then they will have row independence if and only if they are surjective. As $(G,p^n) \rightarrow (G,p)^\phi$ as $n \rightarrow \infty$ then $R(G,p^n) \rightarrow R(G,p)^\phi$ entrywise as $n \rightarrow \infty$. Since the set of surjective maps of $L(X^{V(G)}, \mathbb{R}^{E(G)})$ is an open subset and $R(G,p)^\phi$ is surjective then by Lemma \ref{rigopcont} the result follows.
\end{proof}

\section{Rigidity of \texorpdfstring{$K_4$}{K4} in all normed planes}\label{Rigidity of K4 in all normed planes}

In this section we shall prove the following.

\begin{theorem}\label{keytheorem}
$K_4$ is rigid in all normed planes.
\end{theorem}

This shall follow from Lemma \ref{keylemma1}, Lemma \ref{keylemma2} and Lemma \ref{keylemma3}. We shall consider three separate cases; not strictly convex normed planes (Section \ref{The rigidity of K4 in not strictly convex normed planes}), strictly convex but not smooth normed planes (Section \ref{The rigidity of K4 in strictly convex but not smooth normed planes}), and strictly convex and smooth normed planes (Section \ref{The rigidity of K4 in strictly convex and smooth normed planes}).

\subsection{The rigidity of \texorpdfstring{$K_4$}{K4} in not strictly convex normed planes} \label{The rigidity of K4 in not strictly convex normed planes}

\begin{lemma}\label{notstrictlemma1}
For any $x \in S_1[0] \cap \smooth$ the set $\varphi(x)^{-1}[\{1\}] \cap S_1[0]$ is closed and convex.
\end{lemma}

\begin{proof}
Choose $y, z \in \varphi(x)^{-1}[\{1\}] \cap S_1[0]$, then $\varphi(x) (ty +(1-t)z) =1$ for all $t \in [0,1]$. We further note that
\begin{align*}
1=|\varphi(x) (ty +(1-t)z)| \leq \|ty +(1-t)z \| \leq 1,
\end{align*}
thus $ty +(1-t)z \in S_1[0]$ also and $\varphi(x)^{-1}[\{1\}] \cap S_1[0]$ is convex. As $\varphi(x)$ is continuous then $\varphi(x)^{-1}[\{1\}] \cap S_1[0]$ is closed also. 
\end{proof}

If $\dim X=2$ it follows that $\varphi(x)^{-1}[\{1\}] \cap S_1[0] = [x_1, x_2]$ as $S_1[0]$ is a $1$-dimensional topological manifold homeomorphic to the circle.

\begin{lemma}\label{notstrictlemma2}
If $[x_1,x_2] \subset S_1[0]$ and $x,y \in [x_1,x_2] \cap \smooth$ then $\varphi(x)=\varphi(y)$.
\end{lemma}

\begin{proof}
If $x_1=x_2$ this is immediate so assume $x_1 \neq x_2$. Choose $x := t_0 x_1 + (1-t_0) x_2 \in (x_1,x_2)$ for $t_0 \in (0,1)$ and define the convex and differentiable map $f :[0,1] \rightarrow \mathbb{R}$ where
\begin{align*}
f(t) := \varphi(x)(tx_1 +(1-t)x_2) = t \varphi(x) x_1 +(1-t)\varphi(x)x_2.
\end{align*}
We note $f(t_0)=1$ and $f'(t) = \varphi(x) x_1 - \varphi(x)x_2$, thus if $f$ is not constant then there exists $t \in[0,1]$ where $f(t) >1$; however we note
\begin{align*}
|f(t)| \leq t |\varphi(x)x_1| +(1-t)|\varphi(x)x_2| \leq 1,
\end{align*}
a contradiction. As $f$ is constant then $f(t)=f(t_0) =1$ for all $t \in [0,1]$, thus $\varphi(x)$ is a support functional for all $y \in [x_1,x_2]$ and the result follows.
\end{proof}

\begin{lemma}\label{notstrictlemma3}
Let $x,y \in S_1[0] \cap \smooth$ where $\varphi(x)^{-1}[\{1\}] \cap S_1[0] = [x_1,x_2]$ and $x_1 \neq x_2$. Define $a, b \in \mathbb{R}$ such that $y = a x_1 + b x_2$, then one of the following holds:
\begin{enumerate}[(i)]
\item \label{notstrictlemma3item1} $a, b \geq 0$ or $a, b \leq 0$ and $\varphi(x), \varphi(y)$ are linearly dependent.

\item \label{notstrictlemma3item2} $a<0 <b$ or $b < 0 <a$ and $\varphi(x), \varphi(y)$ are linearly independent.
\end{enumerate}
\end{lemma}

\begin{proof}
(\ref{notstrictlemma3item1}): If $a=0$ then $y=x_2$ or $-x_2$ and $\varphi(x), \varphi(y)$ are linearly dependent; similarly if $b=0$ then $\varphi(x), \varphi(y)$ are linearly dependent. We first note that $\varphi(x)y = a+b$. If $a,b >0$ then
\begin{align*}
a+b = \varphi(x)y \leq \|y\| = \|a x_1 + bx_2 \| \leq a +b,
\end{align*}
thus $\varphi(x)$ is a support functional of $y$. If $a,b<0$ then similarly we have $\varphi(y) = - \varphi(-y)= - \varphi(x)$; in either case $\varphi(x), \varphi(y)$ are linearly dependent.

(\ref{notstrictlemma3item2}): Let $a<0 <b$ and $\varphi(x),\varphi(y)$ be linearly dependent. As $\varphi(y) = - \varphi(-y)$ we may assume $\varphi(y) = \varphi(x)$, thus $\varphi(x)y= 1$. By assumption this implies $y \in [x_1, x_2]$; it follows that there exists $t \in [0,1]$ such that
\begin{align*}
y = t x_1 + (1-t) x_2,
\end{align*}
thus $a,b \geq 0$ contradicting our assumption. We see a similar contradiction if $b < 0 <a$ and $\varphi(x),\varphi(y)$ be linearly dependent, thus the result holds. 
\end{proof}

\begin{lemma}\label{notstrictlemma4}
Let $X$ be a normed plane that is not strictly convex, then there exists $x, y \in S_1[0] \cap \smooth$ such that the following holds:
\begin{enumerate}[(i)]
\item $\varphi(x)^{-1}[\{1\}] \cap S_1[0] = [x_1,x_2]$ with $x_1 \neq x_2$.

\item $\varphi(x), \varphi(y)$ are linearly independent.

\item $y = a x_1 -b x_2$ for $a,b>0$.

\item $-a x_1 + 2b x_2 \in \smooth$.
\end{enumerate}
\end{lemma}

\begin{proof}
By our assumption that $X$ is not strictly convex there exists $x \in S_1[0] \cap \smooth$ such that $\varphi(x)^{-1}[\{1\}] \cap S_1[0] = [x_1,x_2]$ with $x_1 \neq x_2$. By Proposition \ref{abc} and Lemma \ref{notstrictlemma3} there exists $y' \in S_1[0] \cap \smooth$ such that $\varphi(x), \varphi(y')$ are linearly independent and $\varphi(y)= c \varphi(x_1) - d\varphi(x_2)$ for $c,d >0$. If $-c x_1 +2d x_2$ is smooth define $a:=c$, $b := d$ and $y :=y'$. 

Suppose $-c x_1 +2d x_2$ is not smooth. Define the linear isomorphism $T \in L(X)$ where $T(x_1) = -x_1$ and $T(x_2) = 2 x_2$ and $D := T^{-1}(\smooth)$. By part \ref{paper1item1} of Proposition \ref{paper1}, $\smooth^c$ has Lebesgue measure zero. As $T^{-1}$ is linear then $D^c = T^{-1}(\smooth^c)$ must also have Lebesgue measure zero, thus $D \cap \smooth$ is a dense subset in $X$. Since $\varphi$ is continuous if we choose $y \in D \cap \smooth$ sufficiently close to $y'$ then $\varphi(x), \varphi(y)$ will be linearly independent. It then follows by Lemma \ref{notstrictlemma3} that if $y = ax_1 - bx_2$ then $a, b>0$ and by our choice of $y$ we will also have $-a x_1 + 2b x_2 \in \smooth$ as required.
\end{proof}

We define for any $x_1, x_2 \in X$ the following sets:
\begin{enumerate}[(i)]
\item The \textit{open cone},
\begin{align*}
\cone^+ (x_1,x_2) := \{ ax_1 +bx_2 : a,b > 0 \} = \{ rx : x \in (x_1,x_2) , r > 0 \}.
\end{align*}

\item The \textit{closed cone},
\begin{align*}
\cone^+ [x_1,x_2] := \{ ax_1 +bx_2 : a,b \geq 0 \} = \{ rx : x \in [x_1,x_2] , r \geq 0 \}.
\end{align*}

\item The \textit{two-sided open cone},
\begin{align*}
\cone (x_1,x_2) := \cone^+ (x_1,x_2) \cup \cone^+ (-x_1,-x_2).
\end{align*}

\item The \textit{two-sided closed cone},
\begin{align*}
\cone [x_1,x_2] := \cone^+ [x_1,x_2] \cup \cone^+ [-x_1,-x_2].
\end{align*}
\end{enumerate}
If $x_1, x_2$ are linearly independent then the (two-sided) open cone is open and the (two-sided) closed cone is cone.

\begin{lemma}\label{notstrictlemma5}
Let $x_1, x_2 \in S_1[0]$ be linearly independent in a normed plane $X$ and $f \in X^*$ be a support functional of both $x_1$ and $x_2$. Then the following holds:
\begin{enumerate}[(i)]
\item \label{notstrictlemma5item1} If $y \in \cone^+ [x_1,x_2]$ then $\|y\|f$ is a support functional for $y$.

\item \label{notstrictlemma5item2} If $y \in \cone^+ (x_1, x_2)$ then $y$ is smooth.
\end{enumerate}
\end{lemma}

\begin{proof}
(\ref{notstrictlemma5item1}): Let $y \in \cone^+ [x_1,x_2]$. By scaling we may assume $\|y\|=1$, thus $y = tx_1 + (1-t)x_2$ for some $t \in [0,1]$. We now note that
\begin{align*}
f(y) = t f(x_1) + (1-t) f(x_2) =1
\end{align*}
and thus $f$ is a support functional for $y$. 

(\ref{notstrictlemma5item2}): Suppose $y \in \cone^+ (x_1, x_2)$ is not smooth. By scaling we may assume $\|y\|=1$, thus $y = tx_1 + (1-t)x_2$ for some $t \in (0,1)$. As $y$ is not smooth then $y$ has support functional $g \in X^*$ with $f \neq g$. If $g$ isn't a support functional for either $x_1$ or $x_2$ then 
\begin{align*}
g(y) = t g(x_1) + (1-t) g(x_2) < 1,
\end{align*}
thus $g$ is a support functional for both $x_1, x_2$. It follows by \ref{notstrictlemma5item1} that $f,g$ are support functionals for all $x \in \cone^+ (x_1, x_2)$, thus $\cone (x_1, x_2) \subseteq \smooth^c$. As $\cone (x_1, x_2)$ is a non-empty open set this contradicts part \ref{paper1item0} of Proposition \ref{paper1}.
\end{proof}

\begin{lemma}\label{notstrictlemma6}
Let $L$ be a line in a normed plane $X$ that does not contain $0$, then the set $\smooth \cap L$ is dense in $L$.
\end{lemma}

\begin{proof}
Suppose otherwise, then there exists distinct $x_1, x_2 \in L$ and $r>0$ such that $(x_1, x_2)$ lies in $L \setminus \smooth$. We note that $x_1, x_2$ must be linearly independent as $0 \notin L$, thus $\cone^+(x_1 , x_2)$ is a non-empty open subset of $X$. Since $\varphi$ is homogeneous it follows that $\cone^+(x_1 , x_2) \subseteq \smooth^c$ which contradicts part \ref{paper1item0} of Proposition \ref{paper1}.
\end{proof}

We are now ready for our key lemma of the section.

\begin{lemma}\label{keylemma1}
Let $X$ be a normed plane that is not strictly convex, then $K_4$ is rigid in $X$.
\end{lemma}

\begin{proof}
Choose $x, y \in S_1[0] \cap \smooth$ as in Lemma \ref{notstrictlemma4} and let $V(K_4) = \{v_1, v_2, v_3, v_4 \}$. Define for $r >0$ the placement $p^r$ of $K_4$ where:
\begin{enumerate}[(i)]
\item $p^r_{v_1} =0$,

\item $p^r_{v_2} = ax_1 - r y = \left(1-r\right)a x_1 + r b x_2$,

\item $p^r_{v_3} = bx_1 + r y = r a x_1 + \left(1-r\right) b x_2$,

\item $p^r_{v_4} = \left(1- 2 r\right) y = \left(1- 2 r\right) a x_1 - \left(1- 2 r\right) b x_2$.
\end{enumerate}
We note for all $0 < r < \frac{1}{3}$ the following holds:
\begin{enumerate}[(i)]
\item $p^r_{v_2} - p^r_{v_1}$ , $p^r_{v_3} - p^r_{v_1}$, $p^r_{v_2} - p^r_{v_4} \in \cone^+ (x_1, x_2)$.

\item $p^r_{v_2} - p^r_{v_3}$ and $p^r_{v_4} - p^r_{v_1}$ are positive scalar multiples of $y$.

\item $p^r_{v_4} - p^r_{v_3} = \left(1- 3r\right) a x_1 - \left(2- 3r \right) b x_2 \notin \cone [x_1, x_2]$.
\end{enumerate}
Define the line 
\begin{align*}
L := \{ a x_1 - 2b x_2 + 3 r( -a x_1 + bx_2) : r \in \mathbb{R} \},
\end{align*}
then by Lemma \ref{notstrictlemma6} it follows we may choose $r \in \left(0 , \frac{1}{3} \right)$ such that $p^r_{v_4} - p^r_{v_3}$ is smooth. Fix $r$ so that this holds and define $\varphi^r_{v,w}$ to be the support functional of $vw$ in $(K_4,p^r)$. We now note the following holds:
\begin{enumerate}[(i)]
\item $\varphi^r_{ v_2 , v_1}$ , $\varphi^r_{ v_3 , v_1}$, $\varphi^r_{ v_2 , v_4}= \varphi (x)$ (Lemma \ref{notstrictlemma5}).

\item $\varphi^r_{ v_2 , v_3}$, $\varphi^r_{ v_4 , v_1} = \varphi(y)$.

\item $\varphi^r_{ v_4 , v_3} = f$ for some $f \in S_1^*[0]$ where $f, \varphi(x)$ are linearly independent (part \ref{notstrictlemma3item2} of Lemma \ref{notstrictlemma3}).
\end{enumerate}
We now obtain the following rigidity matrix for $R(K_4,p^r)$:
\begin{align*}
  \bbordermatrix{  & \scriptstyle{v_1} & \scriptstyle{v_2}  & \scriptstyle{v_3}  &  \scriptstyle{v_4} \cr	
\scriptstyle{v_1 v_2} &    -\varphi(x) & \varphi(x) & 0           & 0           \cr
\scriptstyle{v_1 v_3} &    -\varphi(x) & 0          & \varphi(x)  & 0           \cr
\scriptstyle{v_1 v_4} &		-\varphi(y) & 0          & 0           & -\varphi(y) \cr
\scriptstyle{v_2 v_3} &		0           & \varphi(y) & -\varphi(y) & 0           \cr
\scriptstyle{v_2 v_4} &		0           & \varphi(x) & 0           & -\varphi(x) \cr
\scriptstyle{v_3 v_4} &		0           & 0          & f           & -f          \cr}
\end{align*}
As $\varphi(x),\varphi(y)$ are linearly independent and $f, \varphi(x)$ are linearly independent then it follows that $R(K_4,p^r)$ has independent rows, thus $(K_4, p^r)$ is independent. Since $K_4$ is independent in $X$ it follows by Corollary \ref{isostaticgraphs} that $K_4$ is isostatic as required.
\end{proof}

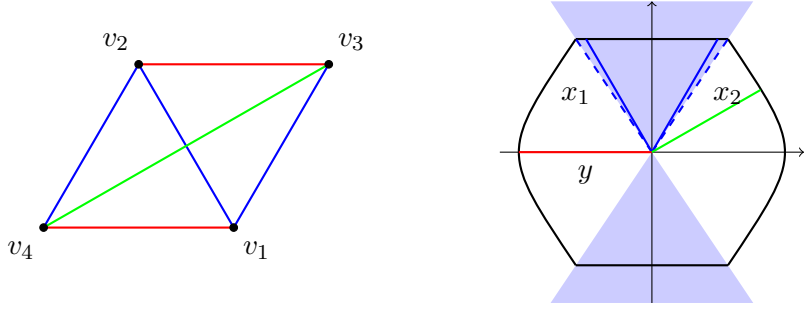
\begin{figure}
\[
\begin{tikzpicture}
\draw[thick, blue] (-5,0) -- (-3.75,2.1651) -- (-2.5,0) -- (-1.25,2.1651); 
\draw[thick, red] (-5,0) -- (-2.5,0);
\draw[thick, red] (-3.75,2.1651) -- (-1.25,2.1651);
\draw[thick, green] (-5,0) -- (-1.25,2.1651);

\node at (-5.3,-0.3) {$v_4$};
\node at (-4.05,2.4651) {$v_2$};
\node at (-2.2,-0.3) {$v_1$};
\node at (-0.95,2.4651) {$v_3$};

\draw[fill] (-5,0) circle [radius=0.05];
\draw[fill] (-3.75,2.1651) circle [radius=0.05];
\draw[fill] (-2.5,0) circle [radius=0.05];
\draw[fill] (-1.25,2.1651) circle [radius=0.05];

\fill[blue, opacity=0.2] (3,1) -- (1.666666666,3) -- (4.33333333,3) -- cycle;
\fill[blue, opacity=0.2] (3,1) -- (1.666666666,-1) -- (4.33333333,-1) -- cycle;

\draw[->] (3,-1) -- (3,3);
\draw[->] (1,1) -- (5,1);

\draw[thick] (2,-0.5) -- (4,-0.5);
\draw[thick] (2,2.5) -- (4,2.5);
\draw[thick] (2,-0.5) .. controls (1,1) .. (2,2.5);
\draw[thick] (4,-0.5) .. controls (5,1) .. (4,2.5);

\draw[thick, blue] (3,1) -- (3.86603,2.5); 
\draw[thick, blue] (3,1) -- (2.13397,2.5); 

\draw[thick, dashed, blue] (3,1) -- (2,2.5); 
\draw[thick, dashed, blue] (3,1) -- (4,2.5);

\draw[thick, red] (3,1) -- (1.25,1);

\draw[thick, green] (3,1) -- ++(30:1.65cm);

\node at (2.125,0.7) {$y$};
\node at (2,1.75) {$x_1$};
\node at (4,1.75) {$x_2$};

\end{tikzpicture}
\]
\caption{A diagram to illustrate Lemma \ref{keylemma1} applied to a not strictly convex normed plane $X$. (Left): The constructed infinitesimally rigid framework $(K_4,p^r)$. (Right): The unit ball of $X$. The edge directions from our placement have been added as their corresponding colour lines, $x_1,x_2$ have been added as blue dashed lines and $\cone[x_1, x_2]$ is shown as the blue area indicated.}
\end{figure}

\subsection{The rigidity of \texorpdfstring{$K_4$}{K4} in strictly convex but not smooth normed planes}\label{The rigidity of K4 in strictly convex but not smooth normed planes}

The following technical lemmas will be of use later. 

\begin{lemma}\label{technicalk4lemma1}
Suppose we have a placement $p$ of a $K_4$ graph with vertices $v_1, v_2, v_3, v_4$ where all edges but $v_1 v_4$ are well-positioned. Further suppose that $\varphi_{v_1, v_2} = \varphi_{v_3, v_4} = \varphi(x)$, $\varphi_{v_1, v_3} = \varphi_{v_2, v_4} = \varphi(y)$ and $\varphi_{v_2 ,v_3} = \varphi(\omega)$ where $\varphi(x), \varphi(y), \varphi(\omega)$ are pairwise independent support functions and $\varphi(\omega) = a \varphi(x) + b \varphi(y)$ for some $a,b \in \mathbb{R}$. Let $\phi$ be the set of support functionals of $(K_4,p)$ with the pseudo-support functional $\varphi_{v_1,v_4}$. If $\varphi_{v_1 ,v_4}$ and $a \varphi(x) - b \varphi(y)$ are linearly independent then $R(K_4, p)^\phi$ has row independence.
\end{lemma}

\begin{proof}
We see that with the given parameters $R(K_4,p)^\phi$ is of the form
\begin{align*}
  \bbordermatrix{  & \scriptstyle{v_1} & \scriptstyle{v_2}  & \scriptstyle{v_3}  &  \scriptstyle{v_4} \cr	
\scriptstyle{v_1 v_2} &    \varphi(x)        & -\varphi(x)& 0          & 0                  \cr
\scriptstyle{v_1 v_3} &    \varphi(y)        & 0          & -\varphi(y)& 0                  \cr
\scriptstyle{v_1 v_4} &		\varphi_{v_1, v_4}& 0          & 0          & -\varphi_{v_1, v_4}\cr
\scriptstyle{v_2 v_3} &		0                 & \varphi(\omega) & -\varphi(\omega)& 0        \cr
\scriptstyle{v_2 v_4} &		0                 & \varphi(y) & 0          & -\varphi(y)        \cr
\scriptstyle{v_3 v_4} &		0                 & 0          & \varphi(x) & -\varphi(x)        \cr}
\end{align*}
Suppose $(c_{vw})_{vw \in E(G)}$ is a pseudo-stress of $(K_4,p)^\phi$. By the second column
\begin{align*}
-c_{v_1 v_2} \varphi(x) + c_{v_2 v_3} \varphi(\omega) + c_{v_2 v_4} \varphi(y) = (c_{v_2 v_3}a -c_{v_1 v_2}) \varphi(x) + (c_{v_2 v_3}b +c_{v_2 v_4}) \varphi(y) =0,
\end{align*}
thus as $\varphi(x), \varphi(y)$ are linearly independent, $c_{v_1 v_2}=c_{v_2 v_3}a$ and $c_{v_2 v_4} = -c_{v_2 v_3}b$. By the third column
\begin{align*}
-c_{v_1 v_3} \varphi(y) - c_{v_2 v_3} \varphi(\omega) + c_{v_3 v_4} \varphi(x) = -(c_{v_2 v_3}a -c_{v_3 v_4}) \varphi(x) - (c_{v_2 v_3}b +c_{v_1 v_3}) \varphi(y) =0,
\end{align*}
thus as $\varphi(x), \varphi(y)$ are linearly independent, $c_{v_3 v_4}=c_{v_2 v_3}a$ and $c_{v_1 v_3} = -c_{v_2 v_3}b$. By the first column combined with our previous results we see that
\begin{align*}
c_{v_1 v_2} \varphi(x) + c_{v_1 v_3} \varphi(y) + c_{v_1 v_4} \varphi_{v_1, v_4} = c_{v_2 v_3}(a \varphi(x) -b \varphi(y)) + c_{v_1 v_4}\varphi_{v_1, v_4} =0.
\end{align*}
Thus as $\varphi_{v_1 ,v_4}$ is linearly independent of $a \varphi(x) - b \varphi(y)$, $c_{v_2 v_3} = c_{v_1 v_4}=0$. This implies $c =0$ and thus $R(K_4,p)^\phi$ has row independence.
\end{proof}

\begin{lemma}\label{technicalk4lemma2}
For all $z \in X$ there exists $x,y \in \smooth$ so that $x +y =z$ and $x-y \in \smooth$. If $z \notin \smooth \cup \{0\}$ then $x,y$ are linearly independent.
\end{lemma}

\begin{proof}
If $z =0$ choose any $x \in \smooth$ and define $y := -x$; similarly if $z \in \smooth$ let $x :=2z$ and $y := -z$. Now suppose $z \notin \smooth \cup \{0\}$. It follows from part \ref{paper1item0} of Proposition \ref{paper1} the sets $z + \smooth$ and $z - \smooth$ have Lebesgue measure zero complements, thus the complement of $(\smooth - z) \cap (\smooth + z)$ has Lebesgue measure zero; it follows that the set is non-empty and we may choose $w \in (\smooth - z) \cap (\smooth + z)$. If we define $x:= \frac{1}{2}(z+w)$ and $y:= \frac{1}{2}(z-w)$ then $x$, $y$, $x-y \in \smooth$ and $z =x+y$. If $x,y$ are linearly dependent then $z$ is smooth, a contradiction, thus $x,y$ are linearly independent. 
\end{proof}

\begin{lemma}\label{technicalk4lemma3}
Let $X$ be a strictly convex normed plane, $z \neq 0$ be non-smooth with $\|z \|=1$, $\varphi[z] =[f,g]$ and define $X^+ := (f-g)^{-1}(0, \infty)$, $X^-:= (f-g)^{-1}(-\infty,0)$. If $(z_n)_{n \in \mathbb{N}}$ is a sequence of smooth points that converges to $z$ with $\| z_n \|=1$, then the following properties hold:
\begin{enumerate}[(i)]
\item \label{tl1} $(\varphi(z_n))_{n \in \mathbb{N}}$ has a convergent subsequence.

\item \label{tl2} If $\varphi(z_n) \rightarrow h$ as $n \rightarrow \infty$ then $h =f$ or $g$.

\item \label{tl3} If $\varphi(z_n) \rightarrow h$ as $n \rightarrow \infty$ and $\varphi(z_n) \in X^+$ for large enough $n$ then $h=f$.

\item \label{tl4} If $\varphi(z_n) \rightarrow h$ as $n \rightarrow \infty$ and $\varphi(z_n) \in X^-$ for large enough $n$ then $h=g$.
\end{enumerate}
\end{lemma}

\begin{proof}
(\ref{tl1}): This holds as $S^*_1[0]$ is compact.

(\ref{tl2}): Choose any $\epsilon >0$, then we may choose $N \in \mathbb{N}$ such that for all $n \geq N$
\begin{align*}
\|h - \varphi(z_n)\| < \frac{\epsilon}{2} \quad \text{ and } \quad \|z - z_n\| < \frac{\epsilon}{2}.
\end{align*}
We now note that $h$ is a support functional for $z$ as $\|h\|=1$ and 
\begin{eqnarray*}
|1 - h(z)| &=& |\varphi(z_n)(z_n) - h(z)| \\
&\leq& |\varphi(z_n)(z_n) - \varphi(z_n)(z)| + |\varphi(z_n)(z) - h(z)| \\
&\leq& \|z_n - z \| + \|\varphi(z_n) - h\| \\
&<& \epsilon,
\end{eqnarray*}
thus $h \in [f,g]$. 

If $h$ lies in the interior of $[f,g]$ then for large enough $n \in \mathbb{N}$ we would have $\varphi(z_n)$ in the interior of $[f,g]$ (with respect to $S^*_1[0]$), thus $\varphi(z_n)$ is a support functional of $z$. If $z \neq z_n$ then we note that $[z, z_n] \in S_1[0]$ as for any $t \in [0,1]$
\begin{align*}
1 = \varphi(z_n)(t z + (1-t)z_n) \leq \|t z + (1-t)z_n\| \leq 1,
\end{align*}
however this contradicts the strict convexity of $X$. If $z=z_n$ then as $z_n$ is smooth $z$ is also smooth, however this contradicts the assumption that $z$ is non-smooth. As the only non-interior points are $f,g$ the result follows.

(\ref{tl3}): Suppose for contradiction that $\varphi(z_n) \rightarrow g$ as $n \rightarrow \infty$. As $f \neq g$ then they must be linearly independent (as otherwise $0 \in [f,g] \subset S^*_1[0]$), thus for each $n \in \mathbb{N}$ there exists $a_n,b_n \in \mathbb{R}$ such that $\varphi(z_n) = a_n f + b_n g$; since $\varphi(z_n) \rightarrow g$ then for large enough $n$ we have that $b_n >0$. We note that if $a_n, b_n  \geq 0$ for large enough $n$ then
\begin{eqnarray*}
\|\varphi(z_n)\| &=& \|a_n f + b_n g\| \\
&\leq& a_n + b_n \\
&=& a_n f(z) + b_n g(z) \\
&=& \varphi(z_n)(z) \\
&\leq& \|\varphi(z_n)\|,
\end{eqnarray*}
thus $\varphi(z_n)$ is a support functional of $z$ which as noted before either contradicts that $X$ is strictly convex or that $z_n$ is smooth and $z$ is non-smooth. Suppose that for large enough $n$ we have $a_n <0 <b_n$. We now note that
\begin{eqnarray*}
\varphi(z_n)(z_n) &=& a_n f(z_n) +b_n g(z_n) \\
&=& a_n (f-g)(z_n) + (a_n +b_n)g(z_n) \\
&<& (a_n +b_n)g(z_n) \quad \text{ as } z_n \in X^+ \\
&\leq& a_n + b_n \\
&=& \|b_n g\| - \|-a_n f\| \\
&\leq& \| a_n f + b_n g\| \\
&=& \|\varphi(z_n)\|
\end{eqnarray*}
which implies $\varphi(z_n)(z_n) <1$ contradicting that $\varphi(z_n)$ is the support functional of $z_n$ and $\| z_n \|=1$. It follows that $\varphi(z_n) \nrightarrow g$, thus $\varphi(z_n) \rightarrow f$ by \ref{tl2}. 

\ref{tl4} now follows by the same method given above.
\end{proof}

We are now ready for our key lemma. 

\begin{lemma}\label{keylemma2}
Let $X$ be a strictly convex normed plane with non-zero non-smooth points, then $K_4$ is rigid in $X$.
\end{lemma}

\begin{proof}
We consider $K_4$ to be the complete graph on the vertex set $\{ v_1, v_2, v_3, v_4 \}$. Let $z$ be a non-zero non-smooth point of $X$ with $\|z\|=1$. By Lemma \ref{technicalk4lemma2}, we can choose smooth linearly independent $x,y \in X$ such that $z=x+y$ and $w :=x-y$ is smooth.

Define the placements $p ,q^k$ of $K_4$ for $k \in \mathbb{Z} \setminus \{0\}$ where
\begin{align*}
p_{v_1} = 0, \quad p_{v_2}= x, \quad p_{v_3}=y, \quad  p_{v_4}=x+y =z,
\end{align*}
and:
\begin{align*}
q^k_{v_1} = 0, \quad q^k_{v_2}= x + \frac{1}{k}x, \quad q^k_{v_3}=y, \quad q^k_{v_4}=x+y + \frac{1}{k}x = z + \frac{1}{k}x.
\end{align*}
By Lemma \ref{wellpos2} there exists for each $k \in \mathbb{Z} \setminus \{0\}$ a well-positioned placement $p^k$ such that $\|p^k - q^k \|_{V(K_4)} < \frac{1}{k^2}$ and $p^k_{v_1} =0$.

By part \ref{paper1item2} of Proposition \ref{paper1}, the support functionals $\varphi_{v,w}^k$ for $p^k$ satisfy the following:
\begin{enumerate}[(i)]
\item \begin{align*}
\lim_{k \rightarrow \infty} \varphi^k_{v_2,v_1} = \lim_{k \rightarrow -\infty} \varphi^k_{v_2,v_1} = \lim_{k \rightarrow \infty} \varphi^k_{v_4,v_3} = \lim_{k \rightarrow -\infty} \varphi^k_{v_4,v_3} = \frac{1}{\|x\|}\varphi(x),
\end{align*}

\item \begin{align*}
\lim_{k \rightarrow \infty} \varphi^k_{v_3,v_1} = \lim_{k \rightarrow -\infty} \varphi^k_{v_3,v_1} = \lim_{k \rightarrow \infty} \varphi^k_{v_4,v_2} = \lim_{k \rightarrow -\infty} \varphi^k_{v_4,v_2} = \frac{1}{\|y\|}\varphi(y),
\end{align*}

\item \begin{align*}
\lim_{k \rightarrow \infty} \varphi^k_{v_2,v_3} = \lim_{k \rightarrow -\infty} \varphi^k_{v_2,v_3} = \frac{1}{\|w\|}\varphi(w).
\end{align*}
\end{enumerate}

By part \ref{supportsetitem2} of Proposition \ref{supportset}, $\varphi[z] = [f,g]$ for some $f \neq g$. We now further define $X^+ := (f-g)^{-1}(0, \infty)$, $X^-:= (f-g)^{-1}(-\infty,0)$. We note that $(f-g)x \neq 0$ (as otherwise $x,z$ are linearly independent), thus without loss of generality we may assume $x \in X^+$. For each $k \in \mathbb{Z} \setminus \{0\}$ define $d_k := p^k_{v_4} -q^k_{v_4}$, then $\|d_k\|< \frac{1}{k^2}$. As
\begin{align*}
(f-g)\left( p^k_{v_4} - p^k_{v_1} \right) = (f-g)\left( z + \frac{1}{k} x + d_k \right) = \frac{1}{k}(f-g)(x) + (f-g)(w_k)
\end{align*}
and $\|f-g\| \leq 2$ it follows that
\begin{align*}
\frac{1}{k}(f-g)(x) - \frac{2}{k^2} \leq (f-g)\left( p^k_{v_4} - p^k_{v_1} \right) \leq \frac{1}{k}(f-g)(x) + \frac{2}{k^2},
\end{align*}
thus there exists $N \in \mathbb{N}$ such that if $k \geq N$ then $p^k_{v_4} - p^k_{v_1} \in X^+$ and if $k \leq -N$ then $p^k_{v_4} - p^k_{v_1} \in X^-$. By part \ref{tl1} of Lemma \ref{technicalk4lemma3} it follows that there exists a strictly increasing sequence $(n_i)_{i \in \mathbb{N}}$ in $\mathbb{N}$ such that 
\begin{align*}
\lim_{i \rightarrow \infty} \varphi^{n_i}_{v_4,v_1} = f \qquad \lim_{i \rightarrow \infty} \varphi^{-n_i}_{v_4,v_1} = g.
\end{align*}

Define $\phi_f$ to be the support functionals of $(K_4,p)$ with pseudo-support functional $\varphi_{v_4,v_1} = f$ and likewise define $\phi_g$ to be the support functionals of $(K_4,p)$ with pseudo-support functional $\varphi_{v_4,v_1} = g$. We note that $R(K_4,p^{n_i}) \rightarrow R(K_4,p)^{\phi_f}$ and $R(K_4,p^{-n_i}) \rightarrow R(K_4,p)^{\phi_g}$ as $i \rightarrow \infty$. 

There exists unique $a,b \in \mathbb{R}$ such that $\varphi(w)= a \varphi(x) +b\varphi(y)$. By Lemma \ref{technicalk4lemma1}, $R(K_4,p)^{\phi_f}$ has row independence if $f$ is linearly independent of $a \varphi(x) -b\varphi(y)$ and $R(K_4,p)^{\phi_g}$ has row independence if $g$ is linearly independent of $a \varphi(x) -b\varphi(y)$. Both $f,g$ cannot be linearly dependent to $a \varphi(x) -b\varphi(y)$ as $f,g$ are linearly independent, thus either $R(K_4,p)^{\phi_f}$ or $R(K_4,p)^{\phi_g}$ has row independence. By Lemma \ref{framelimit} this implies that for large enough $i$ we have either $(K_4,p^{n_i})$ or $(K_4,p^{-n_i})$ are independent and thus there exists an independent placement of $K_4$. It now follows by Proposition \ref{isostatic} that $K_4$ is rigid also.
\end{proof}

\begin{figure}
\[
\begin{tikzpicture}
\draw[thick] (0,0) -- (1,1) -- (0,2) -- (-1,1) -- (0,0);
\draw[thick] (1,1) -- (-1,1);
\draw[dashed, red] (0,0) -- (0,2);

\draw[fill] (0,0) circle [radius=0.05];
\draw[fill] (1,1) circle [radius=0.05];
\draw[fill] (-1,1) circle [radius=0.05];
\draw[fill] (0,2) circle [radius=0.05];

\node at (0,-0.3) {$v_1$};
\node at (-1.3,1) {$v_2$};
\node at (1.3,1) {$v_3$};
\node at (0,2.3) {$v_4$};

\draw[thick] (4,0) -- (5.2,1.2) -- (4.2,2.2) -- (3,1) -- (4,0);
\draw[thick] (5.2,1.2) -- (3,1);
\draw[thick, blue] (4,0) -- (4.2,2.2);

\draw[fill] (4,0) circle [radius=0.05];
\draw[fill] (5.2,1.2) circle [radius=0.05];
\draw[fill] (3,1) circle [radius=0.05];
\draw[fill] (4.2,2.2) circle [radius=0.05];

\draw[->, red, thick] (5.2,1.2) -- (5,1);
\draw[->, red, thick] (4.2,2.2) -- (4,2);

\node at (4,-0.3) {$v_1$};
\node at (2.7,1) {$v_2$};
\node at (5.5,1.2) {$v_3$};
\node at (4.2,2.5) {$v_4$};

\draw[thick] (-4,0) -- (-3.2,0.8) -- (-4.2,1.8) -- (-5,1) -- (-4,0);
\draw[thick] (-3.2,0.8) -- (-5,1);
\draw[thick, green] (-4,0) -- (-4.2,1.8);

\draw[fill] (-4,0) circle [radius=0.05];
\draw[fill] (-3.2,0.8) circle [radius=0.05];
\draw[fill] (-5,1) circle [radius=0.05];
\draw[fill] (-4.2,1.8) circle [radius=0.05];

\draw[->, red, thick] (-3.2,0.8) -- (-3,1);
\draw[->, red, thick] (-4.2,1.8) -- (-4,2);

\node at (-4,-0.3) {$v_1$};
\node at (-5.3,1) {$v_2$};
\node at (-2.9,0.8) {$v_3$};
\node at (-4.2,2.1) {$v_4$};

\end{tikzpicture}
\]
\caption{From left to right: $(K_4,p^{-n_i})$, $(K_4,p)$ and $(K_4,p^{n_i})$ for $i \in \mathbb{N}$. The red dashed edge indicates the edge $v_1 v_4$ of $(K_4,p)$ is not well-positioned. We note that the support functional of the green edge will approximate $g$ while the support functional of the blue edge will approximate $f$.}
\end{figure}
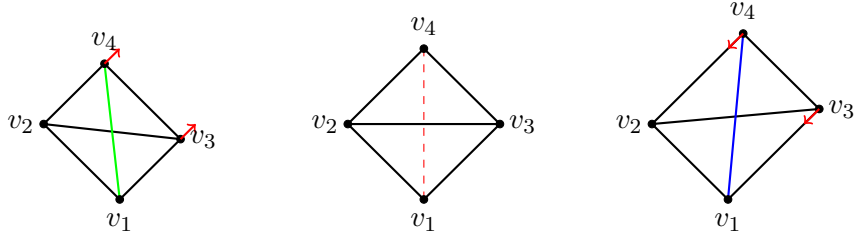

\subsection{The rigidity of \texorpdfstring{$K_4$}{K4} in strictly convex and smooth normed planes}\label{The rigidity of K4 in strictly convex and smooth normed planes}

For this section we shall define $\{v_1, v_2,v_3,v_4\}$ to be the vertex set of $K_4$ and $e := v_1 v_4$. Given a normed plane $X$ we shall fix a basis $b_1,b_2 \in S_1[0]$. 

\begin{definition}
Let $(G,p)$ be a framework in $X$. We say $(G,p)$ is in \textit{3-cycle general position} if every subframework $(H,q) \subset (G,p)$ with $H \cong K_3$ is in general position.
\end{definition}

\begin{lemma}\label{a1}
Let $(K_4-e,p)$ be in 3-cycle general position in a strictly convex normed plane $X$. Then the following holds:
\begin{enumerate}[(i)]
\item \label{a1item1} For all $q \in f^{-1}_{K_4-e}[f_{K_4}(p)]$, $(K_4-e,q)$ is in 3-cycle general position.

\item \label{a1item2} If $(K_4-e,p)$ is well-positioned, then $(K_4-e,p)$ is independent.
\end{enumerate}
\end{lemma}

\begin{proof}
(\ref{a1item1}): Suppose $(K_4-e,q)$ is not in 3-cycle general position, then without loss of generality we may assume $q_{v_1},q_{v_2},q_{v_3}$ lie on a line. By possibly reordering vertices we note that we have 
\begin{align*}
\|q_{v_1} - q_{v_2} \| + \|q_{v_2}-q_{v_3} \| = \| q_{v_1} - q_{v_3}\|.
\end{align*} 
Define $a_{12}= \|p_{v_1} - p_{v_2}\|$, $a_{23} = \|p_{v_2}-p_{v_3}\|$, $x_{12}= (p_{v_1} - p_{v_2})/a_{12}$ and $x_{23}= (p_{v_2} - p_{v_3})/a_{23}$. As $p$ is in general position we note $a_{12},a_{23} > 0$ and $x_{12},x_{23}$ are linearly independent. As $q \in f^{-1}_{K_4}[f_{K_4}(p)]$ then we have that
\begin{align*}
\| a_{12}x_{12} + a_{23}x_{23}\| = \| a_{12}x_{12}\| + \| a_{23}x_{23}\|.
\end{align*} 
We note that $a_{23}/(a_{12} + a_{23}) = 1 - a_{12}/(a_{12} + a_{23})$, thus if we let $t := a_{12}/(a_{12} + a_{23})$ then $t \in (0,1)$ and 
\begin{align*}
\|tx_{12} + (1-t)x_{23}\|= \frac{\| a_{12}x_{12} + a_{23}x_{23}\|}{(a_{12} + a_{23})} = \frac{\| a_{12}x_{12}\| + \| a_{23}x_{23}\|}{(a_{12} + a_{23})} = \| t x_{12}\| + \| (1-t)x_{23}\| =1
\end{align*} 
which contradicts the strict convexity of $X$.

(\ref{a1item2}): Suppose $a \in \mathbb{R}^{E(K_4)\setminus \{e\}}$ is a stress of $(K_4 -e, p)$. By observing the stress condition at $v_1$ we note
\begin{align*}
a_{v_1 v_2} \varphi_{v_1, v_2} + a_{v_1 v_3} \varphi_{v_1,v_3} = 0.
\end{align*}
As $(K_4-e,p)$ is in 3-cycle general position then by part \ref{stritem2} of Proposition \ref{str} it follows $a_{v_1 v_2}= a_{v_1 v_3} =0$. By the same method if we observe the stress condition at $v_4$ we see that $a_{v_2 v_4}= a_{v_3 v_4} =0$. We now see that the stress condition at $v_2$ is
\begin{align*}
a_{v_1 v_2} \varphi_{v_2, v_1} + a_{v_2 v_3} \varphi_{v_2,v_3} + a_{v_2 v_4} \varphi_{v_2,v_4} = a_{v_2 v_3} \varphi_{v_2,v_3} = 0,
\end{align*}
thus $a=0$ and $(K_4-e,p)$ is independent.
\end{proof}

Define for any graph $G$ and vertex $v \in V(G)$ the map
\begin{align*}
f_{G,v} : X^{V(G)} \rightarrow \mathbb{R}^{E(G)} \times X, ~ p \mapsto (f_G(p), p_v);
\end{align*}
it is immediate that $f_{G,v}$ is differentiable at $p$ if and only if $p$ is well-positioned. We note that the kernel of $df_{G,v}(p)$ is exactly the space of infinitesimal flexes $u$ of $(G,p)$ where $u_v =0$. 

\begin{lemma}\label{a2}
Let $X$ be a strictly convex and smooth normed plane and suppose $(K_4-e,p)$ is in 3-cycle general position with $p_{v_1}=0$, then $V(p) := f^{-1}_{K_4-e,v_1}[f_{K_4-e,v_1}(p)]$ is a 1-dimensional compact Hausdorff $C^1$-manifold.
\end{lemma}

\begin{proof}
As $K_4-e$ is connected it that follows $V(p)$ is bounded. As $f_{K_4-e, v_1}$ is continuous then $V(p)$ is closed, thus $V(p)$ is compact; further as $X^{V(K_4)}$ is Hausdorff so too is $V(p)$.

Choose any $q \in V(p)$, then by part \ref{a1item1} of Lemma \ref{a1}, $(K_4-e,q)$ is in 3-cycle general position. By part \ref{a1item2} of Lemma \ref{a1}, $(K_4-e,q)$ is independent, thus for all $q \in V(p)$ we have that $df_{K_4,v_1}(q)$ is surjective i.e.~$p$ is a regular point of $f_{K_4 - e,v_1}$. It now follows from \cite[Theorem 3.5.2(ii)]{manifold} that $V(p)$ is a $C^1$-manifold with dimension $\dim \ker df_{K_4 -e,v_1}(p) =1$.
\end{proof}

We denote by $\mathbb{T}$ the circle group i.e.~the set $\{ e^{i \phi} : \phi \in (-\pi, \pi] \}$ with topology and group operation inherited from $\mathbb{C} \setminus \{0\}$. We note there exists a surjective continuous map $\theta : X \setminus \{0\} \rightarrow \mathbb{T}$ given by
\begin{align*}
x = \lambda b_1 +\mu b_2 \mapsto \frac{\lambda + \mu i}{\sqrt{\lambda^2 + \mu^2}};
\end{align*}
so long as the basis $b_1,b_2 \in X$ is fixed then $\theta$ will be well-defined. We note that if we restrict $\theta$ to $S_1[0]$ then it is a homeomorphism. Let $x, y \in X \setminus \{0\}$ be linearly independent, then $\theta(x)\theta(y)^{-1} = e^{i \phi} \neq \pm 1$; if $\phi \in (0, \pi)$ then we say $x \theta y$ while if $\phi \in (-\pi, 0)$ then we say $y \theta x$.

Choose any two linearly independent points $x,y$ in a normed plane $X$ and define $L(x,y)$ to be the unique line through $x$ and $y$. By abuse of notation we also denote by $L(x,y)$ the unique linear functional $L(x,y):X \rightarrow \mathbb{R}$ where $L(x,y)x= L(x,y)y=1$. We say that $z,z' \in X$ are on opposite sides of the line $L(x,y)$ if and only if $L(x,y)z < 1 < L(x,y)z'$ or vice versa.

\begin{lemma}\label{a3}
Let $X$, $p$ and $V(p)$ be as defined in Lemma \ref{a2}. Define the maps $f,g : V(p) \rightarrow \{-1,1\}$ where 
\begin{align*}
f(q) = 
\begin{cases}
1, & \text{if } q_{v_2} \theta q_{v_3}\\
-1, & \text{if } q_{v_3} \theta q_{v_2}
\end{cases}
\end{align*}
and 
\begin{align*}
g(q) = 
\begin{cases}
1, & \text{if } L(q_{v_2},q_{v_3})(q_{v_4}) >1\\
-1, & \text{if } L(q_{v_2},q_{v_3})(q_{v_4}) <1,
\end{cases}
\end{align*}
then $f,g$ are well-defined and continuous. 
\end{lemma}

\begin{proof}
We note that $f$ is not well-defined at $q$ if and only if $q_{v_2}, q_{v_3}$ are linearly dependent. By part \ref{a1item1} of Lemma \ref{a1}, as $(K_4-e,q)$ is in 3-cycle general position and $q_{v_1} =0$ then $q_{v_2}, q_{v_3}$ are linearly independent, thus $f$ is well-defined at all $q \in V(p)$. 

The map $g$ is not well-defined at $q$ if either $q_{v_2},q_{v_3}$ are linearly dependent or $q_{v_4}$ lies on $L(q_{v_2},q_{v_3})$. By part \ref{a1item1} of Lemma \ref{a1}, as either would imply $(K_4-e,q)$ is not in 3-cycle general position we have that $g$ is well-defined. 

As $f$ and $g$ are locally constant they are continuous.
\end{proof}

\begin{lemma}\cite[Proposition 31]{minkowskigeom}\label{monlemma}
Let $X$ be a strictly convex normed plane and $a,b,c \in X \setminus \{0\}$ be distinct with $\|b\|=\|c\|$. If $a \theta b$, $b \theta c$ and $a \theta c$, or $c \theta b$, $b \theta a$ and $c \theta a$, then $\|a-b\| < \|a-c\|$.
\end{lemma}

\begin{lemma}\label{a5}
Let $X$ be a strictly convex normed plane, $x,y \in X$ be distinct and $r_x, r_y >0$. If $S_{r_x}[x] \cap S_{r_y}[y] \neq \emptyset$ then one of the following holds:
\begin{enumerate}[(i)]
\item \label{a5item1} $S_{r_x}[x] \cap S_{r_y}[y] = \{z\}$ and $x,y,z$ are colinear.

\item \label{a5item2} $S_{r_x}[x] \cap S_{r_y}[y] = \{z_1, z_2\}$ for $z_1 \neq z_2$. Further, if $x =0$ then $z_1 \theta y$ and $y \theta z$ or vice versa, and if $x,y$ are linearly independent then $z_1, z_2$ are on opposite sides of the line $L(x,y)$.
\end{enumerate}
\end{lemma}

\begin{proof}
Let $\theta : S_1[0] \rightarrow \mathbb{T}$ be as previously described. Define the continuous map $\phi : [-\pi , \pi] \rightarrow S_{r_x}[x]$, $\phi(t) := r_x \theta^{-1} (e^{i(t +t_0)}) +x$, where $r_x \theta^{-1} (e^{it_0})$ the unique point between $x,y$ on $S_{r_x}[x]$; we note that $\phi(-\pi)= \phi(\pi)$. Now define the map $h: [-\pi , \pi] \rightarrow \mathbb{R}$, $h(t) := \| \phi(t) - y \|$, then $h(-\pi)=h(\pi)$. It follows from Lemma \ref{monlemma} that $h$ is strictly increasing on $[0, \pi]$ and strictly decreasing on $[-\pi, 0]$. 

If $\phi (0) \in S_{r_x}[x] \cap S_{r_y}[y]$ then for all $t \neq 0$,
\begin{align*}
\|\phi(t) - y \| = h(t) > h(0) = r_y,
\end{align*}
thus $S_{r_x}[x] \cap S_{r_y}[y] = \{z\}$ with $z := \phi(0)$; similarly if $\phi (\pi) \in S_{r_x}[x] \cap S_{r_y}[y]$ then $S_{r_x}[x] \cap S_{r_y}[y] = \{z\}$ with $z := \phi(\pi)$ and so \ref{a5item1} holds.

Suppose $\phi(0), \phi (\pi) \notin S_{r_x}[x] \cap S_{r_y}[y]$. We note that as $S_{r_x}[x] \cap S_{r_y}[y] \neq \emptyset$ then there exists $t_1 \in (-\pi,\pi) \setminus \{0\}$ so that $h(t_1) = r_y$. First suppose $t_1 \in (-\pi,0)$, then for all $t \in (t_1,0)$ and $t' \in (-\pi, t_1)$ we have $h(t) < h(t_1) < h(t')$, thus there are no other intersection points in $(-\pi,0)$. As $h|_{[0,\pi]}$ is strictly increasing and 
\begin{align*}
h(0) < h(t_1) = r_y < h(-\pi) = h(\pi)
\end{align*}
then by the Intermediate Value Theorem there exists a unique value $t_2 \in (0,\pi)$ so that $h(t_2) = r_y$, thus $S_{r_x}[x] \cap S_{r_y}[y] = \{\phi(t_1), \phi(t_2)\}$ with $-\pi<t_1<0<t_2<\pi$. Similarly if $t_1 \in (0,\pi)$ then $S_{r_x}[x] \cap S_{r_y}[y] = \{\phi(t_1), \phi(t_2)\}$ with $-\pi<t_2<0<t_1<\pi$. 

If $x =0$ then it is immediate that $\phi(t_1) \theta \phi(0)$ and $\phi(0) \theta \phi(t_2)$. As $\phi(0)$ is a positive scalar multiple of $y$ then $\phi(t_1) \theta y$ and $y \theta \phi(t_2)$. Now suppose $x,y$ are linearly independent, then we now note that $\phi(t_1)$ and $\phi(t_2)$ lie on opposite sides of the line through $x,y$ as $e^{i(t_1 +t_0)}$ and $e^{i(t_2 +t_0)}$ lie on opposite sides of the line through $e^{i t_0}$ and $e^{-it_0}$.
\end{proof}

\begin{lemma}\label{a5.5}
Let $X$, $p$ and $V(p)$ be as defined in Lemma \ref{a2}. Let $q^1,q^2 \in V(p)$ with $f(q^1)=f(q^2)$, $g(q^1)=g(q^2)$ and $q^1_{v_2}= q^2_{v_2}$, then $q^1=q^2$.
\end{lemma}

\begin{proof}
By part \ref{a1item1} of Lemma \ref{a1}, $q^1,q^2$ are in 3-cycle general position. As $q^1_{v_1},q^1_{v_2}, q^1_{v_3}$ are not colinear then by Lemma \ref{a5} there exists exactly one other point $z \in X$ such that $\|z - q^1_{v_1}\| = \|q^1_{v_3} - q^1_{v_1}\|$ and $\|z - q^1_{v_2}\| = \|q^1_{v_3} - q^1_{v_2}\|$. We note that as $q^1_{v_1}=q^2_{v_1} =0$ and $q^1_{v_2}=q^2_{v_2}$ then $q^2_{v_3}=q^1_{v_3}$ or $q^2_{v_3} =z$. By part \ref{a5item2} of Lemma \ref{a5}, either $z \theta q^1_{v_2}$ and $q^1_{v_3} \theta z$ or vice versa. If $q^2_{v_3} =z$ then $f(q^2)= -f(q^1)$, thus $q^2_{v_3}=q^1_{v_3}$. 

Similarly, as $q^1_{v_2},q^1_{v_3}, q^1_{v_4}$ are not colinear then by Lemma \ref{a5} there exists exactly one other point $z' \in X$ such that $\|z' - q^1_{v_2}\| = \|q^1_{v_4} - q^1_{v_2}\|$ and $\|z' - q^1_{v_3}\| = \|q^1_{v_4} - q^1_{v_3}\|$. By part \ref{a5item2} of Lemma \ref{a5}, $z', q^1_{v_4}$ are on the opposite sides of $L(q^1_{v_2}, q^1_{v_3})$. If $q^2_{v_4} =z'$ then $g(q^2)= -g(q^1)$, thus $q^2_{v_4}=q^1_{v_4}$. 
\end{proof}

\begin{lemma}\label{a6}
Let $X$, $p$ and $V(p)$ be as defined in Lemma \ref{a2}. The path-connected components of $V(p)$ are exactly $f^{-1}[1] \cap g^{-1}[1]$, $f^{-1}[1] \cap g^{-1}[-1]$, $f^{-1}[-1] \cap g^{-1}[1]$ and $f^{-1}[-1] \cap g^{-1}[-1]$. Further, each $f^{-1}[i] \cap g^{-1}[j]$ component is a path-connected compact Hausdorff 1-dimensional $C^1$-manifold.
\end{lemma}

\begin{proof}
By multiple applications of Lemma \ref{a5} it follows that $f^{-1}[1] \cap g^{-1}[1]$, $f^{-1}[1] \cap g^{-1}[-1]$, $f^{-1}[-1] \cap g^{-1}[1]$ and $f^{-1}[-1] \cap g^{-1}[-1]$ are non-empty sets.

Choose $i,j \in \{1,-1\}$. Suppose there exists disjoint path-connected components of $A , B \subset f^{-1}[i] \cap g^{-1}[j]$, then by Lemma \ref{a2}, $A, B$ are both path-connected compact Hausdorff 1-dimensional $C^1$-manifolds. As every path-connected compact Hausdorff 1-dimensional manifold is homeomorphic to a circle (see \cite[Theorem 5.27]{topmanlee}) we may define the homeomorphisms $\alpha : \mathbb{T} \rightarrow A$ and $\beta: \mathbb{T} \rightarrow B$. We will define $\alpha_{v_i}$, $\beta_{v_i}$ to be the $v_i$ component of $\alpha$ and $\beta$ respectively.

Suppose there exists $z_1,z_2 \in \mathbb{T}$ such that $\alpha_{v_2}(z_1) = \alpha_{v_2}(z_2)$, then by Lemma \ref{a5.5}, $\alpha(z_1)= \alpha(z_2)$, thus the map $\alpha_{v_2}: \mathbb{T} \rightarrow S_{\|p_{v_2}\|}[0]$ is injective; similarly, the map $\beta_{v_2}: \mathbb{T} \rightarrow S_{\|p_{v_2}\|}[0]$ is also injective. As $\mathbb{T}$ is compact then by the Brouwer's theorem for invariance of domain \cite[Theorem 1.18]{manifoldlee} it follows $\alpha_{v_2}, \beta_{v_2}$ are homeomorphisms, thus we may choose $z, z' \in \mathbb{T}$ so that $\alpha_{v_2}(z)= \beta_{v_2}(z')$. By Lemma \ref{a5.5} it follows $\alpha(z)=\beta(z')$ and $A,B$ are not disjoint path-connected components. 
\end{proof}

\begin{lemma}\label{a7}
Let $X$, $p$ and $V(p)$ be as defined in Lemma \ref{a2} and $V_0(p)$ be the path-connected component of $V(p)$ that contains $p$. Suppose $p_{v_4}=p_{v_2}+ p_{v_3}$, then for all $q \in V_0(p)$ we have $q_{v_4}=q_{v_2} +q_{v_3}$.
\end{lemma}

\begin{proof}
Choose $q \in V_0(p)$ then by Lemma \ref{a6}, $f(q) = f(p)$ and $g(q) = g(p)$. Define $q'$ to be the placement of $K_4-e$ where $q'_{v_i}=q_{v_i}$ for $i=1,2,3$ and $q'_{v_4}=q'_{v_2} +q'_{v_3}$. We immediately note $q' \in V(p)$ and $f(q')=f(p)$. Suppose $q'\neq q$, then by Lemma \ref{a5.5} we must have $-g(q')= g(q) =g(p)$; however 
\begin{align*}
L(p_{v_2}, p_{v_3})(p_{v_4})= L(p_{v_2}, p_{v_3})(p_{v_2} + p_{v_3})= 2 >1
\end{align*}
and 
\begin{align*}
L(q'_{v_2}, q'_{v_3})(q'_{v_4})= L(q'_{v_2}, q'_{v_3})(q'_{v_2} + q'_{v_3})= 2 >1,
\end{align*}
and so $g(q')=1=g(p)$, a contradiction, thus $q'=q$ and the result follows.
\end{proof}

We will finally need the following result which will help us separate when we are dealing with Euclidean and non-Euclidean normed planes.

\begin{theorem}\cite[p. 323]{alonsoben}\label{a8}
If $X$ is a non-Euclidean normed plane then for all $0< \epsilon <2$ where $\epsilon \neq 2 \cos (k\pi/2n)$ ($n, k \in \mathbb{N}$, $1 \leq k \leq n$),
\begin{align*}
\inf \{\|a+b\|: \|a-b\|= \epsilon, ~ \|a\|=\|b\|=1\} < \sup \{\|a+b\| : \|a-b\|= \epsilon, ~ \|a\|=\|b\|=1\}.
\end{align*}
\end{theorem}

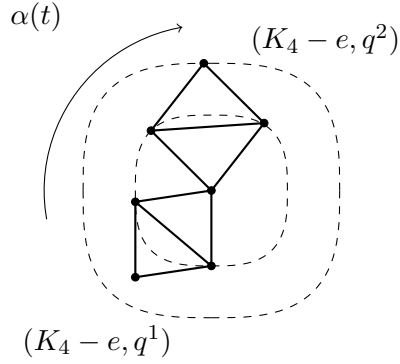
\begin{figure}
\[
\begin{tikzpicture}
\draw[thick] (0,0) -- (-0.794,0.794); 
\draw[thick] (0,0) -- (0.694,0.894); 

\draw[thick] (-0.794,0.794) -- (-0.1,1.687); 
\draw[thick] (-0.794,0.794) -- (0.694,0.894); 
\draw[thick] (0.694,0.894) -- (-0.1,1.687); 

\draw [dashed,domain=0:90] plot ({(cos(\x))^(2/3)}, {(sin(\x))^(2/3)});
\draw [dashed,domain=90:180] plot ({-(-cos(\x))^(2/3)}, {(sin(\x))^(2/3)});
\draw [dashed,domain=180:270] plot ({-(-cos(\x))^(2/3)}, {-(-sin(\x))^(2/3)});
\draw [dashed,domain=270:360] plot ({(cos(\x))^(2/3)}, {-(-sin(\x))^(2/3)});

\draw [dashed,domain=0:90] plot ({1.687*(cos(\x))^(2/3)}, {1.687*(sin(\x))^(2/3)});
\draw [dashed,domain=90:180] plot ({-1.687*(-cos(\x))^(2/3)}, {1.687*(sin(\x))^(2/3)});
\draw [dashed,domain=180:270] plot ({-1.687*(-cos(\x))^(2/3)}, {-1.687*(-sin(\x))^(2/3)});
\draw [dashed,domain=270:360] plot ({1.687*(cos(\x))^(2/3)}, {-1.687*(-sin(\x))^(2/3)});

\node at (1.5,2) {$(K_4-e,q^2)$};

\draw[fill] (0,0) circle [radius=0.05];
\draw[fill] (-0.794,0.794) circle [radius=0.05];
\draw[fill] (0.694,0.894) circle [radius=0.05];
\draw[fill] (-0.1,1.687) circle [radius=0.05];

\draw[thick] (0,0) -- (-1,-0.15); 
\draw[thick] (0,0) -- (0,-1); 

\draw[thick] (-1,-0.15) -- (0,-1); 
\draw[thick] (-1,-0.15) -- (-1,-1.15); 
\draw[thick] (0,-1) -- (-1,-1.15);

\draw[fill] (-1,-0.15) circle [radius=0.05];
\draw[fill] (0,-1) circle [radius=0.05];
\draw[fill] (-1,-1.15) circle [radius=0.05];

\node at (-1.5,-2) {$(K_4-e,q^1)$};

\node at (-2.3,2.3) {$\alpha(t)$};
\draw[<-,domain=100:190] plot ({2.2*cos(\x)}, {2.2*sin(\x)});

\end{tikzpicture}
\]
\caption{The frameworks $(K_4-e ,q^1)$ and $(K_4-e, q^2)$ in some strictly convex and smooth normed plane $X$, as described in Lemma \ref{keylemma3}. The inner dotted shape represents the unit sphere of $X$ and the outer dotted shape represents the sphere of $X$ with radius $\|q^2_{v_4}\|$. As the framework follows the differentiable path $\alpha(t)$ the distance $\| \alpha_{v_1}(t) - \alpha_{v_4}(t) \|$ is non-constant; when the derivative of $t \mapsto \| \alpha_{v_1}(t) - \alpha_{v_4}(t) \|$ is non-zero at point $s$ we add the edge $v_1 v_4$ and note $(K_4, \alpha(s))$ will be infinitesimally rigid.}\label{strconsmoothgraph}
\end{figure}

\begin{lemma}\label{keylemma3}
Let $X$ be a normed plane that is strictly convex and smooth, then $K_4$ is rigid in $X$.
\end{lemma}

\begin{proof}
If $X$ is Euclidean this follows from Theorem \ref{lamanstheorem} so suppose $X$ is non-Euclidean.

Choose any $0 < \epsilon < 2$ so that $\epsilon \neq 2 \cos (k\pi/2n)$ for all $n, k \in \mathbb{N}$ with $1 \leq k \leq n$. By the continuity of the norm we may choose a placement $p$ of $K_4$ so that: 
\begin{enumerate}[(i)]
\item $p_{v_1}=0$,

\item $\|p_{v_2}\|=\|p_{v_3}\| =1$,

\item $p_{v_2} \theta p_{v_3}$,

\item $\|p_{v_2} - p_{v_3}\| = \epsilon$,

\item $p_{v_4}= p_{v_2}+p_{v_3}$,
\end{enumerate}
We note $f(p)=1$ as $p_{v_2} \theta p_{v_3}$, and $g(p)=1$ as 
\begin{align*}
L(p_{v_2}, p_{v_3})(p_{v_4})= L(p_{v_2}, p_{v_3})(p_{v_2} + p_{v_3})= 2 >1.
\end{align*}

We note that $(K_4-e,p)$ is in 3-cycle general position and so by Lemma \ref{a2} and Lemma \ref{a6}, $V_0(p) = f^{-1}[1] \cap g^{-1}[1]$ is a path-connected compact Hausdorff 1-dimensional $C^1$-manifold. We note that for every pair $a,b$ in $S_1[0]$ with $\|a-b\|=\epsilon$ there exists $q \in V_0(p)$ so that $q_{v_2}=a$ and $q_{v_3}=b$ or vice versa, thus there exists $q^1, q^2 \in V_0(p)$ so that
\begin{align*}
\|q^1_{v_4}\| = \inf \{\|a+b\| : \|a-b\|= \epsilon, ~ \|a\|=\|b\|=1\} \\
\|q^2_{v_4}\| = \sup\{\|a+b\| : \|a-b\|= \epsilon, ~ \|a\|=\|b\|=1\};
\end{align*}
further, by Theorem \ref{a8} we have that $\|q^2_{v_4}\|- \|q^1_{v_4}\| >0$. 

As $V_0(p)$ is a path connected $C^1$-manifold that is $C^1$-diffeomorphic to $\mathbb{T}$ we may define a $C^1$-differentiable path $\alpha:[0,1] \rightarrow V_0(p)$ where $\alpha(0)=q^1$, $\alpha(1)=q^2$ and $\alpha'(t) \neq 0$ for all $t \in [0,1]$. By Lemma \ref{a7}, $\alpha_{v_4}(t) = \alpha_{v_2}(t) + \alpha_{v_3}(t)$ for all $t \in [0,1]$; further, as $\alpha_{v_2}(t), \alpha_{v_3}$ are linearly independent, thus $\alpha_{v_4}(t) \neq 0$. As $X$ is smooth, $(K_4,\alpha(t))$ is well-positioned for all $t \in [0,1]$. By part \ref{paper1item1} and part \ref{paper1item3} of Proposition \ref{paper1}, for all $1 \leq i < j \leq 4$, $(i,j) \neq (1,4)$ and $t \in [0,1]$,
\begin{align*}
0=\frac{d}{dt} \|\alpha_{v_i}(t) - \alpha_{v_j}(t)\| = \varphi\left(\frac{\alpha_{v_i}(t) - \alpha_{v_j}(t)}{\|\alpha_{v_i}(t) - \alpha_{v_j}(t)\|} \right)(\alpha'_{v_i}(t) - \alpha'_{v_j}(t)),
\end{align*}
thus $\alpha'(t)$ is a non-trivial flex of $(K_4 - e,\alpha(t))$ with $\alpha'_{v_1}(t)=0$. By part \ref{a1item2} of Lemma \ref{a1}, $(K_4 - e,\alpha(t))$ is independent and so it follows from Theorem \ref{maxwellpaper} that $\alpha'(t)$ is the unique (up to scalar multiplication) non-trivial flex of $(K_4-e,\alpha(t))$ with $\alpha'_{v_1}(t)=0$. By the Mean Value Theorem it follows that there exists $s \in [0,1]$ so that
\begin{align*}
\varphi\left(\frac{\alpha_{v_4}(s)}{\|\alpha_{v_4}(s)\|} \right)(\alpha'_{v_4}(s)) = \frac{d}{dt}\|\alpha_{v_4}(t)\||_{t=s} = \|q^2_{v_4}\| - \|q^1_{v_4}\|>0,
\end{align*}
thus $\alpha'(s)$ is not a flex of $(K_4,\alpha(s))$. As $\mathcal{F}(K_4, \alpha(s)) \subset \mathcal{F}(K_4-e, \alpha(s))$ then $(K_4, \alpha(s))$ is infinitesimally rigid as required.
\end{proof}

\section{Graph operations for the normed plane}\label{Graph extensions in the normed plane}

In this section we shall define a set of graph operations and prove that they preserve isostaticity in non-Euclidean normed planes. The Henneberg moves and the vertex split have also been shown to preserve isostaticity in the Euclidean normed plane and can even be generalised to higher dimensions \cite{comrig} \cite{vertexsplit}, however the vertex-to-$K_4$ extension is strictly a non-Euclidean normed plane graph operation as it will not preserve $(2,3)$-sparsity.

\subsection{0-extensions} \label{0-extensions}

\begin{figure}
\[
\begin{tikzpicture}
\draw (-5,0) ellipse (0.75cm and 0.4cm); 
\draw[fill] (-5.25,0.25) circle [radius=0.05];
\draw[fill] (-4.75,0.25) circle [radius=0.05];

\draw[->, very thick] (-3.7,0) -- (-3.3,0);

\draw (-2,0) ellipse (0.75cm and 0.4cm); 
\draw[fill] (-2.25,0.25) circle [radius=0.05];
\draw[fill] (-1.75,0.25) circle [radius=0.05];
\draw[fill] (-2,0.75) circle [radius=0.05];
\draw (-2.25,0.25) -- (-2,0.75) -- (-1.75,0.25);

\draw (2,0) ellipse (0.75cm and 0.4cm); 
\draw[fill] (2.25,0.25) circle [radius=0.05];
\draw[fill] (1.75,0.25) circle [radius=0.05];
\draw[fill] (2,0) circle [radius=0.05];

\draw (1.75,0.25) -- (2.25,0.25);

\draw[->, very thick] (3.3,0) -- (3.7,0);

\draw (5,0) ellipse (0.75cm and 0.4cm); 
\draw[fill] (5.25,0.25) circle [radius=0.05];
\draw[fill] (4.75,0.25) circle [radius=0.05];
\draw[fill] (5,0) circle [radius=0.05];
\draw[fill] (5,0.75) circle [radius=0.05];

\draw (4.75,0.25) -- (5,0.75) -- (5.25,0.25);
\draw (5,0.75) -- (5,0);

\end{tikzpicture}
\]
\caption{A $0$-extension (left) and a $1$-extension (right).}\label{graph0x}
\end{figure}
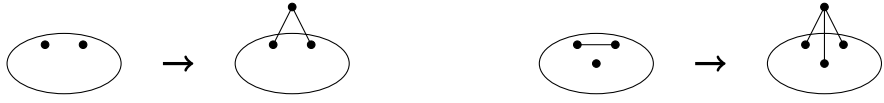

\begin{lemma}\label{zeroextopensmooth}
$0$-extensions preserve independence, dependence and isostaticity in any normed plane.
\end{lemma}

\begin{proof}
Let $G$ be an independent graph in a normed plane $X$. Since we can only apply $0$-extensions to graphs with at least two vertices we may assume that $|V(G)| \geq 2$ and define $v_1, v_2 \in V(G)$ to be the vertices where we are applying the $0$-extension. By Lemma \ref{regopen3} we may choose $p \in \mathcal{R}(G) \cap \mathcal{G}(G)$. Let $G'$ be the $0$-extension of $G$ at $v_1$, $v_2$ with added vertex $v_0$. By Proposition \ref{abc} we may choose linearly independent $y_1, y_2 \in \smooth$ such that $\|y_1\|=\|y_2\|=1$ and $\varphi(y_1), \varphi(y_2) \in X^*$ are linearly independent. Define for $i=1,2$ the lines
\begin{align*}
L_i := \{ p_{v_i} + ty_i : t \in \mathbb{R} \},
\end{align*}
then since $p_{v_1} \neq p_{v_2}$ (as $p$ is in general position) and $y_1,y_2$ are linearly independent then there exists a unique point $z \in L_1 \cap L_2$ and $z \neq p_{v_i}$ for $i=1,2$. Define $p'$ to be the placement of $G'$ that agrees with $p$ on $V(G)$ with $p'_{v_0} = z$.  We recall that $\varphi'_{v,w}$ be the support functional $vw \in E(G')$; it is immediate that if $vw \in E(G)$ then $\varphi'_{v,w}= \varphi_{v,w}$. By possibly multiplying $y_i$ by $-1$ we may assume that $\varphi'_{v_0,v_i} = \varphi(y_i)$ for $i=1,2$.

Choose any stress $a = (a_{vw})_{vw \in E(G')}$ of $(G',p')$, then by observing the stress condition at $v_0$ we note that
\begin{align*}
0 = a_{v_0 v_1} \varphi'_{v_0 , v_1} + a_{v_0 v_2} \varphi'_{v_0, v_2} = a_{v_0 v_1} \varphi(y_1) + a_{v_0 v_2} \varphi(y_2).
\end{align*}
Since $\varphi(y_1), \varphi(y_2)$ are linearly independent then $a_{v_0 v_i}=0$ for $i=1,2$ and $a|_{G}$ is a stress of $(G,p)$. It now follows that there exists a non-zero stress of $(G',p')$ if and only if there exists a non-zero stress of $(G,p)$. By Proposition \ref{equivind} we have that $(G',p')$ is independent if and only if $(G,p)$ is independent, thus $G'$ is independent if $G$ is independent. As $G \subset G'$ then $G$ is independent if $G'$ is independent; this implies that $G'$ is dependent if $G$ is dependent. As $G$ was chosen arbitrarily then it follows that $0$-extensions preserve independence and dependence. 

By Proposition \ref{Henneberg} and Proposition \ref{nonHenneberg}, $(2,k)$-tightness is preserved by $0$-extensions (for $k=2,3$), thus it follows from Corollary \ref{isostaticgraphs} that isostaticity is also preserved.
\end{proof}

\subsection{1-extensions} \label{1-extensions}

\begin{lemma}\label{oneext}
$1$-extensions preserve independence and isostaticity in any normed plane.
\end{lemma}

\begin{proof}
Let $G$ be independent, then as $1$-extensions require 3 vertices we may assume $|V(G)| \geq 3$. We shall suppose $G'$ is a $1$-extension of $G$ that involves deleting the edge $v_1v_2 \in E(G)$ and adding a vertex $v_0$ connected to the end points and some other distinct vertex $v_3 \in V(G)$. By Lemma \ref{regopen3} it follows that there exists a regular (and thus independent) placement $p$ of $G$ in general position. 

By Proposition \ref{abc} there exists $y \in \smooth$, $\|y\|=1$, such that $y , p_{v_1}-p_{v_2}$ are linearly independent and $\varphi(y), \varphi_{v_1,v_2}$ are linearly independent. We note that as $y , p_{v_1}-p_{v_2}$ are linearly independent and $p_{v_1}$, $p_{v_2}$, $p_{v_3}$ are not colinear (since $(G,p)$ is in general position) then the line through $p_{v_1}, p_{v_2}$ and the line through $p_{v_3}$ in the direction $y$ must intersect uniquely at some point $z \neq p_{v_3}$. By parts \ref{paper1item0} and \ref{paper1item2} of Proposition \ref{paper1}, if $z = p_{v_i}$ for some $i=1,2$ we may perturb $y$ to some sufficiently close $y' \in \smooth$ such that the pairs $y' , p_{v_1}-p_{v_2}$ and $\varphi(y'), \varphi_{v_1,v_2}$ are linearly independent and our new intersection point $z'$ is not equal to $p_{v_i}$ for $i=1,2$; we will now assume $y$ is chosen so that this holds. 

Define $p'$ to be the placement of $G'$ where $p'_v = p_v$ for all $v \in V(G)$ and $p'_{v_0} = z$. We recall that $\varphi'_{v,w}$ be the support functional $vw \in E(G')$; it is immediate that if $vw \in E(G) \setminus \{v_1 v_2 \}$ then $\varphi'_{v,w}= \varphi_{v,w}$. We note that $\varphi'_{v_1,v_0}$, $\varphi'_{v_0, v_2}$, $\varphi'_{v_1,v_2}$ are all pairwise linearly dependent, thus there exists $f \in S^*_1[0]$ and $\sigma_{v_i, v_j} \in \{-1,1\}$ such that $\varphi'_{v_i, v_j} = \sigma_{v_i, v_j} f$ for distinct $i,j \in \{0,1,2\}$, with $\sigma_{v_j, v_i} = - \sigma_{v_i, v_j}$. We further note that, due to our choice placement, at least one of $\varphi'_{v_1,v_0}$, $\varphi'_{v_0, v_2}$ must be equal to $\varphi'_{v_1,v_2}$; we may assume by our ordering of $v_1,v_2$ and choice of $f$ that $\sigma_{v_1,v_0} = \sigma_{v_1,v_2}= 1$. We may also assume we chose $y$ such that $\varphi'_{v_0,v_3} = \varphi(y)$ and note that $\varphi(y)$ is linearly independent of $f$ by our choice of $z$.

Choose any stress $a := (a_{vw})_{vw \in E(G')}$ of $(G',p')$. If we observe $a$ at $v_0$ we note  
\begin{align*}
 a_{v_0 v_1} \varphi'_{v_0, v_1} + a_{v_0 v_2} \varphi'_{v_0, v_2} + a_{v_0 v_3} \varphi'_{v_0, v_3} = ( \sigma_{v_0,v_2}a_{v_0v_2} - a_{v_0 v_1})f + a_{v_0 v_3} \varphi(y) =0,
\end{align*}
thus since $f, \varphi(y)$ are linearly independent, $a_{v_0 v_3} = 0$ and $\sigma_{v_0,v_2}a_{v_0 v_1} = a_{v_0v_2}$. Define $b := (b_{vw})_{vw \in E(G)}$ where $b_{vw} = a_{vw}$ for all $vw \in E(G) \setminus \{v_1 v_2\}$ and $b_{v_1 v_2} = a_{v_0 v_1} = \sigma_{v_0,v_2} a_{v_0 v_2}$. For each $v \in V(G) \setminus \{v_1,v_2\}$ it is immediate that
\begin{align*}
\sum_{w \in N_{G}(v)} b_{v w} \varphi_{v, w} = \sum_{w \in N_{G'}(v)} a_{v w} \varphi'_{v, w} =0;
\end{align*}
we note that this will also hold for $v_3$ as $a_{v_0 v_3} = 0$. If we observe whether the stress condition of $b$ holds at $v_1$ we note
\begin{align*}
\sum_{w \in N_{G}(v_1)} b_{v w} \varphi_{v, w} = b_{v_1 v_2} f + \sum_{\substack{w \in N_{G}(v_1) \\ w \neq  v_2}} b_{v w} \varphi_{v, w}= a_{v_0 v_1} \varphi'_{v_1,v_0} + \sum_{\substack{w \in N_{G'}(v_1) \\ w \neq  v_0}} a_{v w} \varphi'_{v, w} = 0,
\end{align*}
while if we observe whether the stress condition of $b$ holds at $v_2$ we note
\begin{align*}
\sum_{w \in N_{G}(v_2)} b_{v w} \varphi_{v, w} = - b_{v_1 v_2} f + \sum_{\substack{w \in N_{G}(v_2) \\ w \neq  v_1}} b_{v w} \varphi_{v, w} = a_{v_0 v_2} \varphi'_{v_2,v_0} + \sum_{\substack{w \in N_{G'}(v_2) \\ w \neq  v_0}} a_{v w} \varphi'_{v, w} = 0,
\end{align*}
thus $b$ is a stress of $(G,p)$. Since $(G,p)$ is independent then $b=0$ which in turn implies $a=0$. As $a$ was chosen arbitrarily then $(G',p')$ is independent; it follows then that $1$-extensions preserve independence.

By Proposition \ref{Henneberg} and Proposition \ref{nonHenneberg}, $(2,k)$-tightness is preserved by $1$-extensions (for $k=2,3$), thus it follows from Corollary \ref{isostaticgraphs} that isostaticity is also preserved.
\end{proof}

\subsection{Vertex splitting}\label{Edge-to-K3 extensions}

A vertex split is given by the following process applied to any graph $G$ (see Figure \ref{graph3}):
\begin{enumerate}
\item Choose an edge $v_0w_0 \in E(G)$, 

\item Add a new vertex $w_0'$ to $V(G)$ and edges $v_0 w_0', w_0 w_0'$ to $E(G)$,

\item For every edge $v w_0 \in E(G)$ we may either leave it or replace it with $v w_0'$.
\end{enumerate}

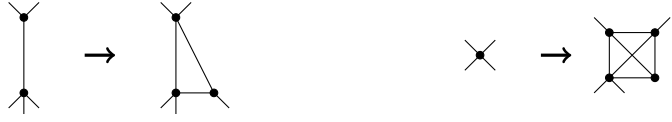
\begin{figure}
\[
\begin{tikzpicture}
\draw[fill] (-4,1) circle [radius=0.05];
\draw[fill] (-4,0) circle [radius=0.05];

\draw (-4,1) -- (-4.2,1.2); 
\draw (-4,1) -- (-3.8,1.2);

\draw (-4,0) -- (-4,1); 

\draw (-4,0) -- (-4.2,-0.2); 
\draw (-4,0) -- (-4,-0.28); 
\draw (-4,0) -- (-3.8,-0.2);

\draw[->, very thick] (-3.2,0.5) -- (-2.8,0.5); 

\draw[fill] (-2,1) circle [radius=0.05];
\draw[fill] (-2,0) circle [radius=0.05];
\draw[fill] (-1.5,0) circle [radius=0.05];

\draw (-2,1) -- (-2.2,1.2); 
\draw (-2,1) -- (-1.8,1.2);

\draw (-2,0) -- (-2,1); 
\draw (-1.5,0) -- (-2,0); 
\draw (-1.5,0) -- (-2,1); 

\draw (-2,0) -- (-2.2,-0.2); 
\draw (-2,0) -- (-2,-0.28); 

\draw (-1.5,0) -- (-1.3,-0.2);

\draw[fill] (2,0.5) circle [radius=0.05];

\draw (2,0.5) -- (2.2,0.7); 
\draw (2,0.5) -- (2.2,0.3); 
\draw (2,0.5) -- (1.8,0.7); 
\draw (2,0.5) -- (1.8,0.3);

\draw[->, very thick] (2.8,0.5) -- (3.2,0.5);

\draw[fill] (4.3,0.8) circle [radius=0.05];
\draw[fill] (4.3,0.2) circle [radius=0.05];
\draw[fill] (3.7,0.8) circle [radius=0.05];
\draw[fill] (3.7,0.2) circle [radius=0.05];

\draw (4.3,0.8) -- (4.3,0.2) -- (3.7,0.8) -- (3.7,0.2);
\draw (3.7,0.8) -- (4.3,0.8) -- (3.7,0.2) -- (4.3,0.2);

\draw (3.7,0.2) -- (3.5,0); 
\draw (3.7,0.2) -- (3.9,0); 
\draw (3.7,0.8) -- (3.5,1); 
\draw (4.3,0.8) -- (4.5,1);

\end{tikzpicture}
\]
\caption{A vertex split (left) and a vertex-to-$K_4$ extension (right).}\label{graph3}
\end{figure}

\begin{lemma}\label{edgemove}
Vertex splitting preserves independence and isostaticity in any normed plane.
\end{lemma}

\begin{proof}
Let $G$ be isostatic, then we may assume that $|V(G)| \geq 3$ and $|E(G)| \geq 1$, as if $|V(G)| =1$ or $|E(G)|=0$ we can't apply a vertex split and if $|V(G)|=2$ we are just applying a $0$-extension. By Lemma \ref{regopen3} we may choose $p$ to be a regular placement of $G$ in general position. Define $G'$ to be graph formed from $G$ by applying a vertex split to $w \in V(G)$ and $v_0 w_0 \in E(G)$ which adds $w'$. We shall define $p'$ to be the not well-positioned placement of $G'$ with $p'_{w_0'}= p'_{w_0} = p_{w_0}$ and $p'_v=p_v$ for all $v \in V(G') \setminus \{w'_0\}$. By Proposition \ref{abc}, we may choose smooth $x \in S_1[0]$ such that $\|x\|=1$, the pair $x, p_{v_0}-p_{w_0}$ are linearly independent, and the pair $\varphi(x)$, $\varphi_{v_0,w_0}$ are linearly independent. We shall define the pseudo-support functional $\varphi'_{w_0,w_0'} := \varphi(x)$ and thus define $(G',p')^\phi$ with $\phi := \{\varphi'_{v,w}: vw \in E(G') \}$.

Let $a := (a_{vw})_{vw \in E(G')}$ be a pseudo-stress of $(G',p')^\phi$. Define $b := (b_{vw})_{vw \in E(G)}$ with $b_{v_0 w_0} = a_{v_0 w_0} + a_{v_0 w_0'}$, $b_{v w_0'}= a_{v w_0}$ if $v \neq v_0$ and $b_{vw} = a_{vw}$ for all other edges of $G$. We shall now show $b$ is a stress of $(G,p)$. We first note that for any $v \in V(G) \setminus \{v_0, w_0\}$ the stress condition of $b$ at $v$ holds as the pseudo-stress of $a$ holds at $v$, and the stress condition of $b$ at $v_0$ holds as 
\begin{align*}
b_{v_0 w_0} \varphi_{v_0,w_0} = a_{v_0 w_0} \varphi_{v_0,w_0} + a_{v_0 w'_0} \varphi'_{v_0,w'_0};
\end{align*}
further, if we observe the stress condition of $b$ at $w_0$ we note
\begin{align*}
\sum_{v \in N_G(w_0)} b_{w_0 v} \varphi_{w_0, v} = \sum_{v \in N_{G'}(w_0) } a_{w_0 v} \varphi_{w_0, v}' + \sum_{v \in N_{G'}(w_0') } a_{w_0' v} \varphi_{w_0', v}' =0 + 0 =0,
\end{align*}
thus $b$ is a stress of $(G,p)$. As $(G,p)$ is independent then $b=0$, thus $a_{vw} = 0$ for all edges $vw \neq w_0 w'_0 , v_0 w_0, v_0 w'_0$ of $G'$, and $a_{v_0 w_0} + a_{v_0 w'_0} =0$. We note by observing the pseudo-stress condition of $a$ at $w_0$,
\begin{align*}
0 = \sum_{v \in N_{G'}(w_0)} a_{w_0 v} \varphi'_{w_0, v} = a_{w_0 w_0'} \varphi'_{w_0, w_0'} + a_{v_0 w_0} \varphi'_{w_0, v_0} = a_{w_0 w_0'} \varphi(x) + a_{v_0 w_0} \varphi_{w_0, v_0},
\end{align*}
thus $a_{v_0 w_0} = a_{w_0 w'_0} =0$; similarly, by observing the pseudo-stress condition of $a$ we note $a_{v_0 w'_0} = 0$. It now follows $a=0$, thus $R(G',p')^\phi$ has row independence.
%
%By Proposition \ref{nonHenneberg}, $G'$ is $(2,2)$-tight; it now follows from the Rank-Nullity theorem that $R(G',p')^\phi$ has row independence if and only if $\ker R(G',p')^\phi$ is the $2$-dimensional space of translational flexes. Let $u \in \ker R(G',p')^\phi$, then we may suppose that $u_{v_0} =0$ by subtracting translational flexes as required. 
%
%If $u_{w}=0$ then 
%\begin{align*}
%u_{w'} \in \ker \varphi_{w,w'} \cap \ker \varphi_{v_0, w'} = \ker \varphi(x) \cap \ker \varphi_{v_0, w} = \{0\}
%\end{align*}
%as $\varphi(x)$, $\varphi_{v_0,w}$ are linearly independent, thus $u_w = u_{w'}$. If $u_w \neq 0$ we note that $\varphi_{w,v_0} = \varphi_{w',v_0}$ and so $u_{w'} = a u_w$ for some $a \in \mathbb{R}$. We now see that
%\begin{align*}
%\varphi_{w,w'}(u_w - u_{w'}) = (1-a)\varphi(x)(u_w) = 0;
%\end{align*}
%however as $\varphi(x), \varphi_{v_0,w}$ are linearly independent then $\varphi(x)(u_w) \neq 0$, thus $a=1$ and $u_w = u_{w'}$. 
%
%Let $v w' \in E(G')$, then we note that since $\varphi_{v,w}= \varphi_{v,w'}$ we have 
%\begin{align*}
%\varphi_{v,w}(u_v - u_w) = \varphi_{v,w'}(u_v - u_{w'}) = 0.
%\end{align*}
%This gives that $u|_{V(G)} \in \mathcal{F}(G,p)$. As $(G,p)$ is isostatic, $u_{v_0}=0$ and $X$ is a non-Euclidean normed plane then by part \ref{paperiso1item2} of Proposition \ref{paperiso1}, $u_v = 0$ for all $v \in V(G)$. Since $w \in V(G)$ and $u_{w'} = u_w$ then $u=0$, thus $R(G',p')^\phi$ has row independence.

Define $q^n \in X^{V(G')}$ to be the placement of $G'$ that agrees with $p'$ on $V(G)$ with $q^n_{w'_0} = p'_{w_0} -\frac{1}{n} x$. By Lemma \ref{wellpos2} we may choose $p^n \in \mathcal{W}(G)$ such that $\|p^n -q^n \|_{V(G)} < \frac{1}{n}$, $p^n_{w_0}= q^n_{w_0}$ and $p^n_{w'_0}= q^n_{w'_0}$. By our choice of $p^n$ we have that $\varphi^n_{w_0,w'_0} \rightarrow \varphi'_{w_0,w'_0}$ as $n \rightarrow \infty$, and by part \ref{paper1item2} of Proposition \ref{paper1}, $\varphi^n_{v,w} \rightarrow \varphi'_{v,w}$ as $n \rightarrow \infty$ for all $vw \in E(G') \setminus \{w_0w_0'\}$. This implies $(G',p^n) \rightarrow (G',p')^\phi$ as $n \rightarrow \infty$ and so by Proposition \ref{framelimit} we thus have that $G'$ is independent also. 

Suppose $G$ is isostatic, then by Corollary \ref{isostaticgraphs}, $G$ is $(2,k)$-tight for $k=2$ if $X$ non-Euclidean and $k=3$ if $X$ is Euclidean. By Proposition \ref{nonHenneberg}, $G'$ is $(2,k)$-tight, thus by Corollary \ref{isostaticgraphs}, $G'$ isostatic as required.
\end{proof}

\subsection{Vertex-to-$K_4$ extensions}\label{Vertex-to-K4 extensions}

The vertex-to-$K_4$ extension is given by the following process applied to any graph $G$ (see Figure \ref{graph3}):
\begin{enumerate}
\item Choose a vertex $v_0 \in V(G)$, 

\item Add the vertices $v_1,v_2, v_3, v_4$ to $V(G)$ and edges $v_i v_j$ to $E(G)$, $1 \leq i < j \leq 4$,

\item Delete $v_0$ and replace any edge $v_0 w \in E(G)$ with $v_i w$ for some $i=1,2,3,4$.
\end{enumerate}

\begin{lemma}\label{vertexmove}
Vertex-to-$K_4$ moves preserve isostaticity in any non-Euclidean normed plane.
\end{lemma}

\begin{proof}
By Theorem \ref{keytheorem} and Corollary \ref{isostaticgraphs}, $K_4$ is isostatic in any non-Euclidean normed plane.

Let $G$ be independent in the normed plane $X$ with regular (and thus independent) placement $p$ in general position (Lemma \ref{regopen3}). Let $v_0 \in V(G)$, $K$ be the complete graph with vertices $v_1, v_2,v_3,v_4$ and $G'$ be the graph formed by performing a vertex-to-$K_4$ at $v_0$ by adding vertices $v_1,v_2,v_3,v_4$. We define $p'$ to be the not well-positioned placement of $G'$ that agrees with $p$ on $V(G)$ and has $p'_{v_i} = p_{v_0}$ for all $i=1,2,3,4$. Since $K \cong K_4$ is isostatic we may define an isostatic placement $x:=(x_{v_i})_{i=1}^4$ of $K$ in general position (Lemma \ref{regopen3}) and define the pseudo-support functionals
\begin{align*}
\varphi'_{v_i,v_j} := \varphi \left( \frac{x_{v_i} - x_{v_j}}{\|x_{v_i} - x_{v_j}\|} \right)
\end{align*}
for $1 \leq i <j \leq 4$; by this we may define $\phi$ and $(G',p')^\phi$.

Let $a:=(a_{vw})_{vw \in E(G')}$ be a pseudo-stress of $(G',p')^\phi$. Define $b := (b_{vw})_{vw \in E(G)}$ with $b_{vw} = a_{vw}$ for all $vw \in E(G) \cap E(G')$ and $b_{v_0 w} = a_{v_i w}$ for all $v_i w \in E(G')$ with $w \neq v_j$, $i,j \in \{1,2,3,4\}$. It is immediate that for any vertex $v \in V(G) \setminus N_G(v_0)$ the stress condition of $b$ at $v$ holds. If we observe the stress condition of $b$ at $v_0$ we note
\begin{align*}
\sum_{w \in N_G(v_0)} b_{v_0 w} \varphi_{v_0, w} = \sum_{i=1}^4 \sum_{w \in N_G(v_i)} a_{v_i w} \varphi'_{v_i, w} = \sum_{i=1}^4 0 = 0
\end{align*}
as the internal stress vectors $a_{v_i v_j} \varphi'_{v_i,v_j}$ cancel each other out, thus the stress condition of $b$ at $v_0$ holds and $b$ is a stress of $(G,p)$ is independent. As $(G,p)$ is independent then $b=0$, thus $a_{vw}=0$ for all $vw \neq v_i v_j$ for some $1 \leq i <j \leq 4$. Since $(K,x)$ is independent it follows that $a_{v_i v_j} =0$ for all $1 \leq i <j \leq 4$, thus $a =0$ and $(G',p')^\phi$ is independent.

%By Proposition \ref{nonHenneberg}, $G'$ is $(2,2)$-tight; it now follows from the Rank-Nullity theorem that $R(G',p')^\phi$ has row independence if and only if $\ker R(G',p')^\phi$ is the $2$-dimensional space of translational flexes. Choose any element $u' \in \ker R(G',p')^\phi$, then by subtracting translational flexes we may assume that $u'_{v_1}=0$. As $(K,(p_{v_i})_{i=1}^4)$ is infinitesimally rigid and $X$ is a non-Euclidean normed plane then by part \ref{paperiso1item2} of Proposition \ref{paperiso1}, $u'_{v_i} =0$ for all $i=2,3,4$. 
%
%Define $u \in X^{V(G)}$ such that $u$ agrees with $u'$ on $V(G) \setminus \{v_0\}$ and $u_{v_0} =0$. Choose $vw \in E(G)$. If $v,w \neq v_0$ then 
%\begin{align*}
%\varphi_{v,w}(u_v -u_w)=\varphi_{v,w}(u'_v -u'_w)=0
%\end{align*}
%as $vw \in E(G')$, $p'|_G =p$ and $u' \in \ker R(G',p')^\phi$. If (without loss of generality) $w = v_0$ then for some $i=1,2,3,4$ we have $v v_i \in E(G')$ and
%\begin{align*}
%\varphi_{v,v_0}(u_v - u_{v_0}) = \varphi_{v,v_i}(u'_v - u'_{v_i}) = 0
%\end{align*}
%as $p'_{v_i} =p_{v_0}$, $u'_{v_i}=0$ and $u' \in \ker R(G',p')^\phi$. It follows that $u \in \mathcal{F}(G,p)$, thus as $(G,p)$ is isostatic, $u_{v_0} = 0$ then and $X$ is a non-Euclidean normed plane then by part \ref{paperiso1item2} of Proposition \ref{paperiso1}, $u=0$; this now gives that $\ker R(G',p')^\phi$ is just the space of translational flexes and thus $R(G',p')^\phi$ has row independence.

Define $q^n$ to be the placement of $G'$ where $q^n$ agrees with $p'$ on $V(G') \setminus \{v_1,v_2,v_3,v_4 \}$ and $q^n_{v_i} = p_{v_0} + \frac{1}{n} x_{v_i}$. By Lemma \ref{wellpos2}, we may choose $p^n \in \mathcal{W}(G)$ such that $\|p^n -q^n \|_{V(G)} < \frac{1}{n}$ and $p^n_{v_i} =q^n_{v_i}$ for all $i=1, 2,3,4$. By our choice of $p^n$ we have that $\varphi^n_{v_i,v_j} = \varphi'_{v_i,v_j}$ for $1 \leq i < j \leq 4$, and by part \ref{paper1item2} of Proposition \ref{paper1}, $\varphi^n_{v,w} \rightarrow \varphi'_{v, w}$ as $n \rightarrow \infty$ for all $vw \in E(G') \setminus E(K)$. This implies $(G',p^n) \rightarrow (G',p')^\phi$ as $n \rightarrow \infty$ and so by Proposition \ref{framelimit}, $G'$ is independent. 

Suppose $G$ is isostatic, then by Corollary \ref{isostaticgraphs}, $G$ is $(2,2)$-tight. By Proposition \ref{nonHenneberg}, $G'$ is $(2,2)$-tight, thus by Corollary \ref{isostaticgraphs}, $G'$ isostatic as required.
\end{proof}

\section{Proof of Theorem \ref{laman2} and connectivity conditions for rigidity}\label{Generalised Laman and Lovasz and Yemini theorems}

\subsection{Proof of Theorem \ref{laman2} and immediate corollaries} \label{Generalised Laman's theorem}

We are now ready to prove our main theorem.

\begin{proof}[Proof of Theorem \ref{laman2}]
Suppose $|V(G)| \leq 2$, then $G$ is either $K_1$, $K_2$ or $K_1 \sqcup K_1$ (the graph on 2 vertices with no edges). We note all three are $(2,2)$-sparse but only $K_1$ is $(2,2)$-tight. As $K_1$ and $K_1 \sqcup K_1$ have no edges then both are independent. It is immediate that any well-positioned placement of $K_2$ is independent, thus $K_2$ is also independent. By part \ref{papernecessary2item1} of Theorem \ref{papernecessary2}, both $K_2$ and $K_1 \sqcup K_1$ are infinitesimally flexible while $K_1$ is infinitesimally rigid as required.

Let $G$ be isostatic with $|V(G)| \geq 3$, then by part \ref{papernecessary2item3} of Theorem \ref{papernecessary2}, $G$ is $(2,2)$-tight. 

Now let $G$ be $(2,2)$-tight with $|V(G)| \geq 3$, then by Proposition \ref{nonHenneberg} it can be obtained from $K_4$ by a finite sequence of 0-extensions, 1-extensions, vertex splitting and vertex-to-$K_4$ extensions. By Theorem \ref{keytheorem} and Corollary \ref{isostaticgraphs} $K_4$ is isostatic and so by Lemma \ref{zeroextopensmooth}, Lemma \ref{oneext}, Lemma \ref{edgemove} and Lemma \ref{vertexmove}, $G$ is isostatic. 
\end{proof}

We now have an immediate corollary.

\begin{corollary}
A graph is rigid in all normed planes if and only if it contains a proper $(2,3)$-tight spanning subgraph.
\end{corollary}

\begin{proof}
Let $G$ contain a proper $(2,3)$-tight spanning subgraph $H$. As $H$ is proper there exists $e \in E(G) \setminus E(H)$; it follows that $H+e$ is $(2,2)$-tight spanning subgraph of $G$. By Theorem \ref{lamanstheorem} and Theorem \ref{laman2}, $G$ is rigid in all normed planes.

Suppose $G$ is rigid in all normed planes, then by Theorem \ref{lamanstheorem}, $G$ contains a $(2,3)$-tight spanning subgraph $H$ and $|E(G)| \geq 2|V(G)| -2$. Since $|E(H)| < |E(G)|$ then there exists $e \in E(G) \setminus E(H)$, thus $H$ is proper.
\end{proof}

We note that there exist $(2,2)$-tight graphs which are not rigid in the Euclidean plane, e.g.~consider two copies of $K_4$ joined at a single vertex (see Figure \ref{threeconnect}).

\subsection{Analogues of Lov\'{a}sz \& Yemini's theorem for non-Euclidean normed planes} \label{Generalised Lovasz and Yemini results}

We say that a connected graph is \textit{$k$-connected} if $G$ remains connected after the removal of any $k -1$ vertices and \textit{$k$-edge-connected} if $G$ remains connected after the removal of any $k -1$ edges. This section shall deal with how we may obtain sufficient conditions for rigidity from the connectivity of the graph. The first result is the famous connectivity result given by Lov\'{a}sz \& Yemini in \cite{lovasz}.

\begin{theorem}\label{lovasz}
Any $6$-connected graph is rigid in the Euclidean plane.
\end{theorem}

The following is a corollary of a famous result of Nash-Williams \cite[Theorem 1]{nashwilliams}.

\begin{corollary}\label{nashwilliams}
The following properties hold:
\begin{enumerate}[(i)]
\item $G$ is $(k,k)$-tight if and only if $G$ contains $k$ edge-disjoint spanning trees $T_1, \ldots, T_k$ where $E(G) = \cup_{i=1}^k E(T_i)$ 

\item If $G$ is $k$-edge-connected then $G$ contains $k$ edge-disjoint spanning trees.
\end{enumerate}
\end{corollary}

Using Corollary \ref{nashwilliams} we may obtain an analogous result.

\begin{theorem}\label{lovasz2}
Any $4$-edge-connected graph is rigid in all non-Euclidean normed planes.
\end{theorem}

\begin{proof}
By Corollary \ref{nashwilliams} if $G$ is $4$-edge-connected then it will contain two edge-disjoint spanning trees, thus by Corollary \ref{nashwilliams}, $G$ must have a $(2,2)$-tight spanning subgraph $H$. By Theorem \ref{laman2} we have that $G$ is rigid in any non-Euclidean normed plane as required.
\end{proof}

Since $k$-connectivity implies $k$-edge-connectivity then we can see that a $4$-connected graph will also be rigid in all non-Euclidean normed planes. We note that this is the best possible result as we can find graphs that are $3$-edge-connected but do not contain a $(2,2)$-tight spanning subgraph (see Figure \ref{threeconnect}).

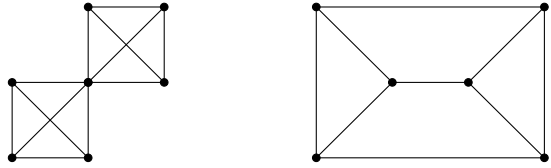
\begin{figure}
\[
\begin{tikzpicture}

\draw[fill] (8,1) circle [radius=0.05];
\draw[fill] (8,-1) circle [radius=0.05];
\draw[fill] (7,0) circle [radius=0.05];
\draw[fill] (6,0) circle [radius=0.05];
\draw[fill] (5,-1) circle [radius=0.05];
\draw[fill] (5,1) circle [radius=0.05];

\draw (5,1) -- (5,-1) -- (8,-1) -- (8,1) -- (5,1) -- (6,0) -- (7,0) -- (8,1);
\draw (5,-1) -- (6,0);
\draw (8,-1) -- (7,0);

\draw[fill] (2,0) circle [radius=0.05];
\draw[fill] (1,0) circle [radius=0.05];
\draw[fill] (2,-1) circle [radius=0.05];
\draw[fill] (1,-1) circle [radius=0.05];

\draw (2,0) -- (1,0) -- (1,-1) -- (2,-1) -- (2,0);
\draw (2,0) -- (1,-1);
\draw (1,0) -- (2,-1);

\draw[fill] (2,0) circle [radius=0.05];
\draw[fill] (3,0) circle [radius=0.05];
\draw[fill] (2,1) circle [radius=0.05];
\draw[fill] (3,1) circle [radius=0.05];

\draw (2,0) -- (3,0) -- (3,1) -- (2,1) -- (2,0);
\draw (2,0) -- (3,1);
\draw (3,0) -- (2,1);

\end{tikzpicture}
\]
\caption{(Left): A $(2,2)$-tight graph that is not rigid in the Euclidean plane. (Right): A $3$-connected (and hence $3$-edge-connected) graph that does not contain a $(2,2)$-tight spanning subgraph.}\label{threeconnect}
\end{figure}

\begin{corollary}
Any $6$-connected graph is rigid in all normed planes.
\end{corollary}

\begin{proof}
As $G$ is $6$-connected then by Theorem \ref{lovasz}, $G$ is rigid in the Euclidean normed plane. As $6$-connected implies 6-edge-connected then $G$ is $4$-edge-connected, thus by Theorem \ref{lovasz2}, $G$ is rigid in any non-Euclidean normed plane.
\end{proof}

This following result is generalisation of Lov\'{a}sz \& Yemini's theorem given by Tibor Jord\'{a}n on the number of rigid spanning subgraphs contained in a graph.

\begin{theorem}\cite[Theorem 3.1]{jordanconnect}\label{jordan}
Any $6k$-connected graph contains $k$ edge-disjoint $(2,3)$-tight spanning subgraphs.
\end{theorem} 

Yet again we may obtain an analogous result.

\begin{theorem}\label{jordan2}
Any $4k$-edge-connected graph contains $k$ edge-disjoint $(2,2)$-tight spanning subgraphs.
\end{theorem}

\begin{proof}
By Corollary \ref{nashwilliams} if $G$ is $4k$-edge-connected then it will contain $2k$ edge-disjoint spanning trees, thus by Corollary \ref{nashwilliams}, $G$ has $k$ $(2,2)$-tight spanning subgraphs.
\end{proof}

Combining this we have the final generalisation.

\begin{corollary}\label{jordan3}
Any $6k$-connected graph contains $k$ edge-disjoint spanning subgraphs $H_1, \ldots, H_k$ that are rigid in any normed plane.
\end{corollary} 

\begin{proof}
Since $6k$-connected implies $6k$-edge-connected then by Theorem \ref{jordan} there exists $k$ edge-disjoint $(2,3)$-tight spanning subgraphs $A_1, \ldots, A_k$ and by Theorem \ref{jordan} $k$ edge-disjoint $(2,2)$-tight spanning subgraphs $B_1, \ldots, B_k$. We shall define $A := \cup_{i=1}^k A_i$ and $B := \cup_{i=1}^k B_i$, then $|E(B)|-|E(A)| = k$ and so we may choose $e_1, \ldots, e_k \in E(B) \setminus E(A)$. For any $i,j=1, \ldots,k$ we note that $H_i := A_i +e_i$ will be a $(2,2)$-tight spanning subgraph that contains a $(2,3)$-tight spanning subgraph $A_i$, thus by Theorem \ref{lamanstheorem} and Theorem \ref{laman2}, $H_i$ is rigid in all normed planes. We now note $E(H_i) \cap E(H_j) = \emptyset$ as required.
\end{proof}

\begin{remark}
Corollary \ref{jordan3} only gives that for any normed plane $X$ a graph $G$ will contain $k$ edge-disjoint spanning subgraphs $H_1, \ldots, H_k$ with infinitesimally rigid placements $(H_1,p^1)$, $\ldots$, $(H_k,p^k)$ in $X$. In general this does not guarantee the existence of a single placement $p$ of $G$ such that $(H_1,p), \ldots, (H_k,p)$ are infinitesimally rigid in $X$. However if $\mathcal{R}(H)$ is dense in $\mathcal{W}(H)$ for any subgraph $H \subset G$ then such a placement does exist. An example where this occurs would be any graph in any smooth $\ell_p$ space (see \cite[Lemma 2.7]{infinite}). In contrast, if $X$ has a polyhedral unit ball then this property does not hold in general (see \cite[Lemma 16]{polyhedra}).
\end{remark}


\begin{thebibliography}{1}

%\bibitem{mechanics} R. Abraham, J. E. Marsden, {\em Foundations of mechanics, second edition}, Addison-Wesley, Reading Massachusetts, 1978.


\bibitem{alonsoben} J. Alonso, C, Ben\'{i}tez, {\em Some characteristic and non-characteristic properties of inner product spaces}, Journal of approximation theory, Volume 55 Issue 3, pp 318-325, 1988.


\bibitem{euclidean} D. Amir, {\em Characterization of inner product spaces}, Birkhauser Verlag Basel, Switzerland, 1986.

\bibitem{asiroth} L. Asimow, B. Roth, {\em The rigidity of graphs}, Transactions of the American Mathematical Society Volume 245, pp. 279-289, 1978.

\bibitem{asiroth2} L. Asimow, B. Roth, {\em The rigidity of graphs II}, Journal of Mathematical Analysis and Applications Volume 68 Issue 1, pp. 171-190, 1979.

\bibitem{rotation} J. Cook, J. Lovett, F. Morgan, {\em Rotations in a normed plane}, The American Mathematical Monthly, Volume 114 Issue 7, pp 628-632, 2007.

\bibitem{mypaper} S. Dewar, {\em Equivalence of continuous, local and infinitesimal rigidity in normed spaces}, arXiv.org preprint, https://arxiv.org/abs/1809.01871, 2018.


%\bibitem{nonlinear} Y. Benyamini, J. Lindenstrauss, {\em Geometric Nonlinear Functional Analysis}, Volume 1, American Mathematics Society, 2000.

%\bibitem{connellybook} R. Connelly, {\em Chapter 2: Basic concepts}, unpublished chapter, 1987.

%\bibitem{gluck} H. Gluck, {\em Almost all simply connected closed surfaces are rigid}, Geometric topology (Proc. Conf., Park City, Utah, 1974), Lecture Notes in Mathematics Vol 438, pp. 225-239, Springer, Berlin, 1975.

\bibitem{comrig} J. Graver, B. Servatius, H. Servatius, {\em Combinatorial rigidity}, Graduate Studies in Mathematics volume 2, American Mathematics Society, 1993.

%\bibitem{bchall} B. C. Hall, {\em An elementary introduction to groups and representations}, arXiv.org online lecture notes, https://arxiv.org/abs/math-ph/0005032, 2000.

\bibitem{jordanconnect} T. Jord\'{a}n, {\em On the existence of $k$ edge-disjoint $2$-connected spanning subgraphs}, Journal of Combinatorial Theory Series B Volume 95 Issue 2 pp 257-262, 2005.


\bibitem{polyhedra} D. Kitson, {\em Finite and infinitesimal rigidity with polyhedral norms}, Discrete \& Computational Geometry Volume 54 Issue 2 pp 390-411, Springer US, 2015.

\bibitem{matrixnorm} D. Kitson, R. H. Levene, {\em Graph rigidity for unitarily invariant matrix norms}, arXiv.org preprint, \,  https://arxiv.org/abs/1709.08967, 2017.


\bibitem{noneuclidean} D. Kitson, S. C. Power, {\em Infinitesimal rigidity for non-Euclidean bar-joint frameworks}, Bulletin of the London Mathematical Society Volume 46 Issue 4 pp 685-697, 2014.

\bibitem{infinite} D. Kitson, S. C. Power, {\em The rigidity of infinite graphs}, Discrete Comput Geom Volume 60 Issue 3 pp 531-537, Springer US, 2018.


\bibitem{maxwell} D. Kitson, B. Schulze, {\em Maxwell-Laman counts for bar-joint frameworks in normed spaces}, Linear Algebra and its Applications 481 pp 313-329, 2015.

\bibitem{laman} G. Laman, {\em On graphs and rigidity of plane skeletal structures}, Journal of Engineering Mathematics Volume 4 Issue 4, pp 331-340, 1970.

%\bibitem{streinu} A. Lee, I. Streinu, {\em Pebble game algorithms and sparse graphs}, Discrete mathematics Volume 308 Issue 8, pp. 1425-1437, 2008.



\bibitem{manifoldlee} J. Lee, {\em Manifolds and differential geometry}, Graduate studies in mathematics Volume 107, American Mathematical Society, 2010.

\bibitem{topmanlee} J. Lee, {\em Introduction to topological manifolds, second edition}, Graduate Texts in Mathematics, Springer-Verlag New York, 2011.


\bibitem{lovasz} L. Lov\'{a}sz, Y. Yemini, {\em On generic rigidity in the plane}, SIAM Journal of Algebraic Discrete Methods Volume 3 Issue 1, pp 91-98, 1982.


\bibitem{minkowskigeom} H. Martini, K. J. Swanepoel, Gunter Wei{\ss}, {\em The geometry of Minkowski spaces - A Survey}, Part 1, Expositiones Mathematicae Volume 19  Issue 2 pp 97-142, Urban \& Fischer Verlag, 2001.


\bibitem{manifold} J. E. Marsden, T. Raitu, R. Abraham, {\em Manifolds, tensor analysis, and applications, Third Edition}, Springer-Verlag New York, 2002.



%\bibitem{megginson} R. E. Megginson, {\em Introduction to Banach space theory}, Graduate Texts in Mathematics Volume 183, Springer Science \& Business Media, 1998.

%\bibitem{spheres} D. Montgomery, H. Samelson, {\em Transformation groups of spheres}, Annals of Mathematics, Second Series, Vol. 44, No. 3, pp. 454-470, 1943.

%\bibitem{twotwo} A. Nixon, J.C. Owen, S.C. Power, {\em Rigidity of Frameworks Supported on Surfaces}, SIAM J. Discrete Math. Volume 26 Issue 4 pp 1733-1757, 2012.

\bibitem{twotwoparttwo} A. Nixon, J.C. Owen, {\em An inductive construction of $(2,1)$-tight graphs}, Contributions to Discrete Mathematics Volume 9 Issue 2 pp 1-16, 2014.

\bibitem{nashwilliams} C. St.J. A. Nash-Williams, {\em Edge-disjoint spanning trees of finite graphs}, Journal of the London Mathematical Society Volume 36 Issue 1 pp 445-450, Oxford University Press, 1961.



\bibitem{minkowski} A. C. Thompson, {\em Minkowski geometry}, Encyclopedia of Mathematics and its Applications, Cambridge University Press, 1996.

%\bibitem{convex} C. Niculescu, L. Persson, {\em Convex Functions and their Applications: A Contemporary Approach}, Springer-Verlag New York, 2006.

%\bibitem{convex2} R. R. Phelps, {\em Convex Functions, Monotone Operators and Differentiability}, Volume 1364 of Lecture Notes in Mathematics , Springer, 2013.


%\bibitem{rudin} W. Rudin, {\em Principals of Mathematical Analysis, Third Edition}, International Series in Pure \& Applied Mathematics, McGraw-Hill Higher Education, 1976.






%\bibitem{lie} B. Hall, {\em Lie Groups, Lie Algebras, and Representations: An Elementary Introduction}, Graduate Texts in Mathematics Volume 222, Springer Science \& Business Media, 2003.

%\bibitem{metric} D. Burago, Y. Burago, S. Ivanov, {\em A course in Metric Geometry}, Graduate Studies in Mathematics Volume 33, American Mathematics Society, 2001.





%\bibitem{jordanbook} T. Jord\'{a}n, {\em Combinatorial rigidity: graphs and matroids in the theory of rigid frameworks}, Technical report TR-2014-12, Egerv\'{a}ry Research Group, Budapest, 2014.

%\bibitem{whitely} W. Whiteley, {\em Some matroids from discrete applied geometry}, Contemporary Mathematics, AMS Volume 197, pp 171-311, 1996.


%\bibitem{smoothapprox} R. Schneider, {\em Smooth approximations of convex bodies}, Rendiconti del Circolo Matematico di Palermo, Volume 33 Issue 3, pp 436-440, 1984.







\bibitem{vertexsplit} W. Whiteley, {\em Vertex splitting in isostatic frameworks}, Université du Québec à Montréal, 1990.


\end{thebibliography}
\end{document}